\documentclass[twoside,11pt]{article}

\usepackage{jmlr2e}
\usepackage{algorithm}
\usepackage{algorithmic}

\usepackage{cleveref}

\usepackage{multirow}
\usepackage{multicol}

\usepackage{color}
\usepackage{booktabs}
\usepackage{tabu}

\newtheorem{assumption}{Assumption}

\ShortHeadings{SVRG for SDP}{Zeng, Zha, Ma, and Yao}
\firstpageno{1}

\begin{document}

\title{On Stochastic Variance Reduced Gradient Method for \\Semidefinite Optimization}

\author{\name Jinshan Zeng \email jinshanzeng@jxnu.edu.cn \\
       \addr School of Computer and Information Engineering and Institute of Artificial Intelligence, Jiangxi Normal University, Nanchang, China
       \AND
       \name Yixuan Zha \email zyx@jxnu.edu.cn \\
       \addr School of Computer and Information Engineering, Jiangxi Normal University, Nanchang, China
       \AND
       \name Ke Ma\thanks{Corresponding author} \email  make@iie.ac.cn \\
       \addr Institute of Computing, Chinese Academy of Sciences, and University of Chinese Academy of Sciences, Beijing, China
       \AND
       \name Yuan Yao$^*$ \email  yuany@ust.hk \\
       \addr Department of Mathematics, Chemical and Biomolecular Engineering, Hong Kong University of Science and Technology, Hong Kong.
       }

\editor{xxx}

\maketitle

\begin{abstract}
The low-rank stochastic semidefinite optimization has attracted rising attention due to its wide range of applications. An effective way to solve this problem is based on the low-rank factorization, which significantly improves the computational efficiency but brings some great challenge to the analysis due to the recast nonconvex reformulation. Among existing methods based on the nonconvex reformulation, the stochastic variance reduced gradient (SVRG) method has been regarded as one of the most effective methods. SVRG in general consists of two loops, where a reference full gradient is firstly evaluated in the outer loop and then used to yield a variance reduced estimate of the current gradient in the inner loop. Two options have been suggested to yield the output of the inner loop, where 
Option I
sets the output as its last iterate, 
and Option II
yields the output via selecting randomly from all the iterates in the inner loop. However, there is a significant gap between the theory and practice of SVRG when adapted to the stochastic semidefinite programming (SDP). SVRG practically works better with Option I, while most of existing theoretical results focus on Option II. In this paper, we fill this gap via exploiting a new semi-stochastic variant of the inner loop in the original SVRG, adapted to the semidefinite optimization. Equipped with this, we establish the global linear submanifold convergence (i.e., converging exponentially fast to a submanifold of a global minimum under the orthogonal group action) of the proposed SVRG method with Option I, given a provable initialization scheme and under certain smoothness and restricted strongly convex assumptions.
Our analysis includes the effects of the mini-batch size and update frequency in the inner loop as well as two practical step size strategies, the fixed step size and stabilized Barzilai-Borwein step sizes. Some numerical results in matrix sensing demonstrate that the proposed SVRG method is similar to the original SVRG with Option I, and outperforms their counterparts with Option II as well as stochastic gradient descent and factored gradient descent methods. The effects of algorithmic parameters are also studied numerically. Not only this paper establishes the convergence of SVRG with Option I for semi-definite optimization, but also the new technical development in the convergence analysis is of value to general low-rank matrix learning problems. 
\end{abstract}


\section{Introduction}

There are many applications in scientific research and engineering involving the following low-rank stochastic convex semidefinite optimization problem:
\begin{equation}
\label{Eq:SD-stoch}
\mathop{\mathrm{min}}_{X \in \mathbb{S}^{p}} \ f(X) = \frac{1}{n} \sum_{i=1}^n f_i(X)  \quad \mathrm{s.t.} \quad X \succeq 0,
\end{equation}
where $f_i(X)$ is some convex, smooth cost function associated with the $i$-th sample, $X \succeq 0$ is the positive semidefinite (PSD) constraint, and $\mathbb{S}^{p}$ represents the set of real symmetric matrices of size $p\times p$. In many related applications, certain low-rank assumption is usually imposed on $X$.
Typical applications include the non-metric multidimensional scaling \citep{Agarwal2007,Borg-NMDS2005}, matrix sensing \citep{Sanghavi2013,Recht-2016}, phase retrieval \citep{Candes-ACHA2015}, synchronization \citep{Bandeira-Synchronization2016}, and community detection \citep{Montanari-CommunityDetection2016}.

The classical algorithms for solving problem \eqref{Eq:SD-stoch} mainly include the first-order methods such as the well-known projected gradient descent method \citep{Nestrov-1989}, interior point method \citep{Alizadeh-IPM1995}, and more specialized path-following interior point methods
(for more detail, see the nice survey paper \citep{Monteiro-SDP2003} and references therein).
However, most of these methods are not well-scalable due to the PSD constraint, i.e., $X \succeq 0$.
To break the hurdle of PSD constraint,
the idea of low-rank factorization was adopted in the literature \citep{Monteiro2003,Monteiro2005} and became very popular in the past five years due to its empirical success \citep{Sanghavi-FGD2016,Jin2016,Ma2018,Ma2019,Wang2017-ICML,Zeng-svrg2018,Zeng-SGD2019}.

Mathematically, let $X=UU^T$ for some rectangular matrix $U\in \mathbb{R}^{p\times r}$, $X^*$ be a global optimum of problem \eqref{Eq:SD-stoch} with rank $r$ (where $r<p$), and $g(U):= f(UU^T)$, then problem \eqref{Eq:SD-stoch} can be recast as
\begin{equation}
\label{Eq:SD-determ-ncvx}
\mathop{\mathrm{min}}_{U \in \mathbb{R}^{p\times r}} \ g(U):= f(UU^T).
\end{equation}
Since the PSD constraint has been eliminated, the recast problem (\ref{Eq:SD-determ-ncvx}) has a significant advantage over (\ref{Eq:SD-stoch}), but this benefit has a corresponding cost: the objective function is no longer convex but instead \textit{nonconvex} in general. This brings a great challenge to the analysis. Even for the simple first-order methods like the factored gradient descent (FGD) and gradient descent (GD), its \emph{local linear convergence}
remains unspecified until the recent work in \citep{Sanghavi-FGD2016,Wang2017-GD}, respectively.
Moreover, facing the challenge in large scale applications with a big $n$, stochastic algorithms \citep{Robbins-SGD1951} have been widely adopted nowadays, that is, at each iteration, we only use the gradient information of one or a small batch of the whole sample instead of the full gradient over $n$ samples. However, due to the existence of variance of such stochastic gradients, the stochastic gradient descent (SGD) method either converges to an $O(\eta)$ neighborhood of a global optimum at a linear rate when adopting a fixed step size $\eta$, or has a \textit{local sublinear} convergence rate when adopting a diminishing step size, in the restricted strongly convex (RSC) case \citep{Zeng-SGD2019}.

To accelerate SGD,
various variance reduction techniques have been proposed in the literature, e.g. the stochastic variance reduced gradient (SVRG) method in \citep{Johnson-Zhang-svrg2013} and stochastic average gradient (SAG) method in \citep{Bach-SAG}, which resume the linear convergence for strongly convex problems in Euclidean space. 
One distinction of the SVRG proposed by \cite{Johnson-Zhang-svrg2013} lies in that it does not need to memorize the gradients over training sample set; instead it involves two loops of iterations for updating full-sample gradients and stochastic variations. Its algorithmic implementations give rise to two choices: Option I, where one naturally sets the output of the inner loop as its last iterate; Option II, where one selects the output of the inner loop from all the iterates in the inner loop in a uniformly random way, more expensive in memory cost on recording inner loop iterates. Although the paper \citep{Johnson-Zhang-svrg2013} established Option II's linear convergence guarantee, it leaves open the convergence of Option I, in spite of a more natural and efficient scheme in practice. This gap was later filled by \citep{Tan-Ma-bbsvrg2016} with an introduction of Barzilai-Borwein (BB) step size \citep{BB-stepsize1988}.


When adapted to the the nonconvex matrix problem, linear convergence results are only known for Option II. \cite{Zhang2017-SVRG} studied the SVRG with Option II 
for the matrix sensing problem with the square loss and established its linear \emph{submanifold} convergence, (i.e., converges exponentially fast to a submanifold as the orbit of a global optimum under the orthogonal group action) under certain restricted isometry property (RIP). This work was generalized by 
\cite{Wang2017-ICML} with a variant of SVRG with Option II called \textit{Stochastic Variance-Reduced Gradient for Low-Rank Matrix Recovery} (dubbed \textit{SVRG-LR} for short) for the low-rank matrix recovery problem with general loss and the Restricted Strongly Convex (RSC) condition. 

Yet, for Option I scheme in nonconvex low-rank matrix recovery, convergence theories remain largely open. In fact, Option I is more natural and usually enjoys a faster empirical loss reduction than Option II used in SVRG (see \cite{Johnson-Zhang-svrg2013,Tan-Ma-bbsvrg2016,Sebbouh-Bach-svrg-opt1-19} for example in Euclidean space and also Figure \ref{Fig:comp-option} when adapted to the matrix sensing case considered in our later numerical experiments in Section \ref{sc:experiment}). Moreover, the memory storage required in the implementation of SVRG with Option I (called \textit{SVRG-I} henceforth) is generally much less than that of SVRG with Option II (called \textit{SVRG-II} henceforth), since it only needs to store one iterate of the inner loop when implementing SVRG-I, instead of all iterates of the inner loop required in the implementation of SVRG-II. Therefore, it is desired to establish the convergence of Option I and its variants due to these advantages.   

Particularly when the matrix variable $X$ is positive semi-definite, several recent studies provided partial answers to the convergence problem of Option I. 
First, \cite{Ma2018,Ma2019} adapted the original SVRG with Option I 
to the application of ordinal embedding \citep{Agarwal2007} based on the low-rank factorization reformulation.
In these works, a generalization of the well-known Barzilai-Borwein (BB) step size \citep{BB-stepsize1988} called the \textit{stabilized Barzilai-Borwein (SBB)} step size was also introduced for SVRG to alleviate the challenge of tuning the step size parameter, where the BB step size is a special case of SBB step size.
Under certain smoothness assumption, the convergence to a critical point of the SVRG with Option I and SBB step size was firstly established in \cite{Ma2018,Ma2019}; then its local linear point convergence (i.e., exponentially fast convergence to a global minimum in Euclidean distance starting from a neighborhood of this global minimum) was further established in \cite{Zeng-svrg2018} under the RSC condition. However, the submanifold convergence of SVRG with Option I is still left open, an important problem as a global optimum is invariant under the orthogonal group action. 



\begin{figure}[!t]
\begin{minipage}[b]{0.99\linewidth}
\centering
\includegraphics*[scale=.6]{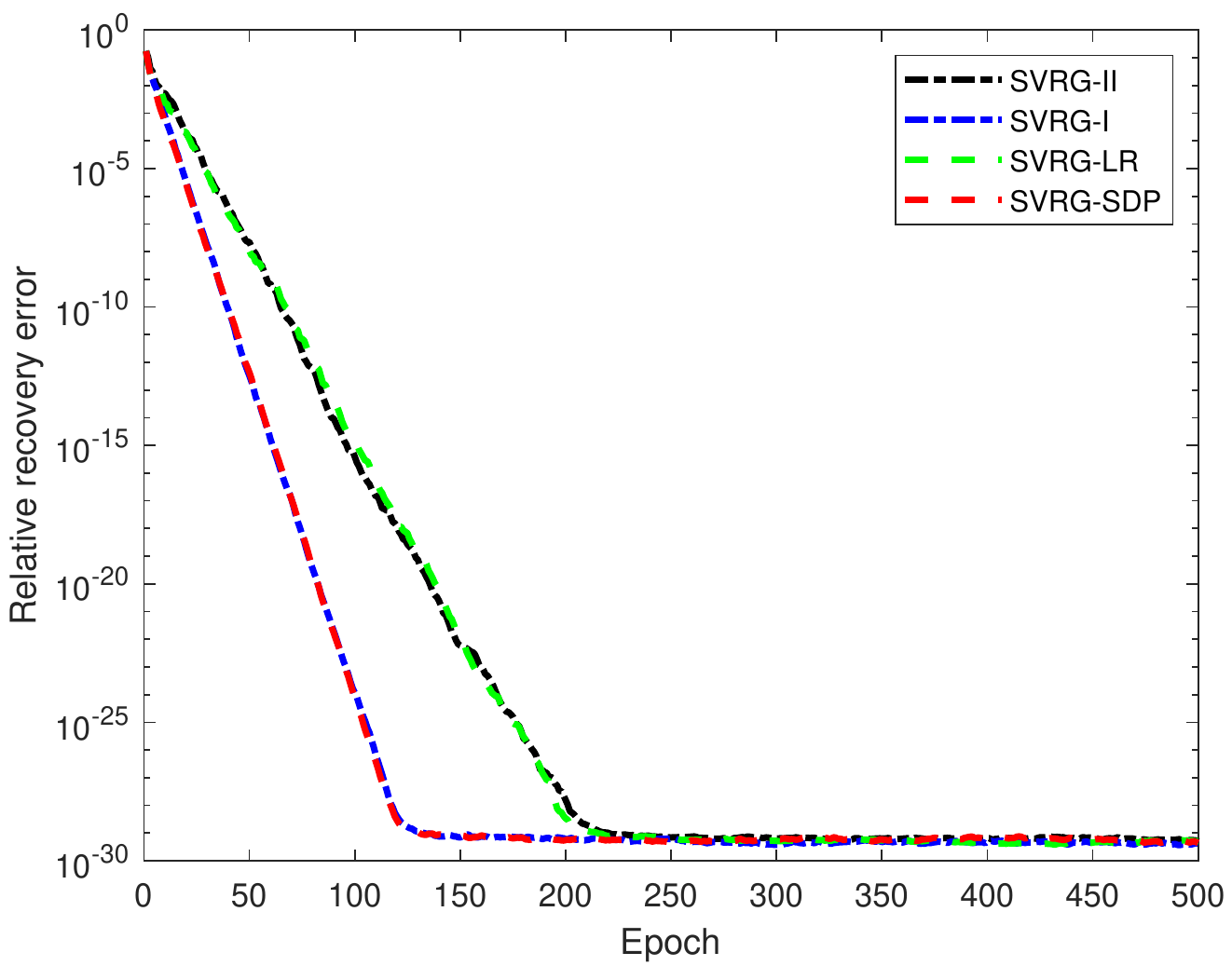}
\end{minipage}
\hfill
\caption{Comparison on the performance of SVRG methods with two options for the low-rank matrix sensing problem, where SVRG-I and SVRG-II are respectively the original SVRG with Option I and II, and SVRG-LR and SVRG-SDP are respectively the modified SVRG using a semi-stochastic gradient with Option II and I. The specific settings can be found in Section \ref{sc:experiment}. It can be observed that SVRG methods equipped with Option I is superior to their Option II counterparts.
}
\label{Fig:comp-option}
\end{figure}

In this paper, we fill this gap by providing a \emph{global linear submanifold} convergence theory of Stochastic Variance-Reduced Gradient method for Semi-definite optimization \eqref{Eq:SD-determ-ncvx}, using a variant of Option I, under general loss and RSC with the fixed step size and general stabilized Barzilai-Borwein step size. To achieve this, motivated by \cite{Wang2017-universal-SVRG}, we adopted the following new \emph{semi-stochastic gradient} (i.e., $\nabla f_{i_t}(X^t)U^t - (\nabla f_{i_t}(\tilde{X}^k) - \nabla f(\tilde{X}^k))U^t$) in place of the original one in the inner loop of SVRG (i.e., $\nabla f_{i_t}(X^t)U^t - (\nabla f_{i_t}(\tilde{X}^k)-\nabla f(\tilde{X}^k))\tilde{U}^k$, called SVRG-I in this paper), see Algorithm \ref{alg:SVRG} for detail. Thus, our algorithm is called \textit{Stochastic Variance-Reduced Gradient for SDP} (dubbed \textit{SVRG-SDP} for short). 
Distinguished to SVRG-LR, we adopt Option I instead of Option II by selecting the output of inner loop as the last iterate ($\tilde{U}^{k+1} = U^m$), as well as the mini-batch strategy. Under the regular Lipschitz gradient and restricted strongly convex assumptions, we at first establish a local linear submanifold convergence of the proposed SVRG method (see Theorem \ref{Thm:svrg}), where the radius of the initial ball permitted to guarantee the linear submanifold convergence is larger than existing results established in the literature \citep{Sanghavi-FGD2016,Zhang2017-SVRG,Wang2017-universal-SVRG,Zeng-svrg2018,Zeng-SGD2019}, as listed in Table \ref{Table:comparison-theory}.
Then, we boost such local linear submanifold convergence to the global linear submanifold convergence\footnote{that is, exponentially fast convergence to a submanifold of the orthogonal group action orbit of a global minimum starting from an initial point that is not necessarily close to the global minimum submanifold.} (see Theorem \ref{Thm:svrg-globalconv}) with an appropriate initial scheme based on the projected gradient descent method (see Algorithm \ref{alg:initialization}).
Moreover, we establish the same global linear submanifold convergence results for the proposed SVRG methods equipped with two practical step size schemes, i.e., the fixed step size and certain stabilized Barzilai-Borwein step size introduced in \cite{Ma2018,Ma2019}.
A series of experiments in matrix sensing are conducted to show the effectiveness of the proposed method. The numerical results show that the proposed method performs similarly to original SVRG with Option I yet with provable convergence guarantee, and works remarkably better than their counterparts with Option II, i.e., SVRG-LR and SVRG-II, as well as the factored gradient descent and stochastic gradient descent methods. The effects of the mini-batch size, update frequency of the inner loop, and step size schemes are also discussed and studied in both theory and experiment.

In contrast to the existing analysis for SVRG with Option II, the main difficulty in dealing with Option I is to yield the contraction between two successive iterates in the outer loop, since the iterates in the inner loop may change largely during the update of Option I. Such possibly dramatic changes among iterates in the inner loop can be alleviated in Option II where the output of the inner loop is a uniformly random sample of iterates, by taking expectation or average over the iterates of the inner loop in convergence analysis. Hence it is relatively easy to tackle the case of Option II in analysis. However when we switch to Option I  where the output of the inner loop is set as the last iterate, one has to carefully develop some new technique to handle the change in the inner loop to avoid being out of control.


Our key development in the convergence analysis is a novel \emph{second-order descent lemma} (see Lemma \ref{Lemm:2nd-order-descent}) for SVRG-SDP about the progress made by a single iterate of the inner loop, characterized by the submanifold metric (i.e., ${\cal E}(U,U_r^*) := \mathop{\min}_{R \in {\cal O}} \|U-U_r^*R\|_F^2$, where $U_r^*$ is a global minimum with rank $r$, ${\cal O}$ is the set of orthogonal matrices of size $r\times r$). This improves the second-order descent lemma in \cite[Lemma 1]{Zeng-svrg2018} with Euclidean metric (i.e., $\tilde{\cal E}(U,U_r^*) := \|U-U_r^*\|_F^2$) and leads to the desired linear submanifold convergence of SVRG-SDP.

\begin{table}
\tabulinesep=1mm
\centering
\caption{\small{Comparisons on linear convergence guarantees between this paper and existing literature,
where $\kappa$ is the ``condition number'' of the objective function $f$ (to be specified in \eqref{Eq:kappa}), $\sigma_r(X^*)$ and $\tau(X^*)$ are respectively the $r$-th largest singular value and the condition number of the optimum $X^*$, and $r=\mathrm{rank}(X^*)$.
Particularly, $\kappa =1$ in the square loss case as considered in \cite{Zhang2017-SVRG}.
The larger bound on ${\cal E}(\tilde{U}^0,U_r^*)$ implies a larger initial ball permitted to guarantee the linear convergence.
$\dag$
The linear convergence of SVRG-I established in \cite{Zeng-svrg2018} is in the sense of point convergence.
From this table, the initial ball permitted to guarantee the linear submanifold convergence of SVRG-SDP established in this paper is slightly larger than existing ones.
}
}
\hspace{.1mm}
\begin{tabu}{lccc}\toprule
Algorithm  &loss & assumption  & ${\cal E}(\tilde{U}^0,U_r^*)$ \\ \hline
FGD (\cite{Sanghavi-FGD2016}) & general & RSC & $\frac{\sigma_r(X^*)}{10000\kappa^2 \tau^2(X^*)}$\\ \hline
SGD (\cite{Zeng-SGD2019}) & general & RSC & $\frac{(\sqrt{2}-1)\sigma_{r}(X^*)}{2\kappa}$ \\ \hline
SVRG-II (\cite{Zhang2017-SVRG}) & square & RIP  &$\frac{\sigma_r(X^*)}{16}$ \\ \hline
SVRG-I (\cite{Zeng-svrg2018})$^\dag$ & general & RSC  &$\frac{(2+\sqrt{3})\cdot(\sqrt{2}-1)\sigma_{r}(X^*)}{6\kappa}$ \\ \hline
SVRG-LR (\cite{Wang2017-universal-SVRG}) & general & RSC & $\frac{2\sigma_{r}(X^*)}{15\kappa}$  \\ \hline
{\bf SVRG-SDP (this paper)} & general & RSC  & $\frac{8(\sqrt{2}-1)\sigma_r(X^*)}{9\kappa}$ \\\bottomrule
\end{tabu}
\label{Table:comparison-theory}
\end{table}


The rest of this paper is organized as follows. In Section \ref{sc:algorithms}, we introduce the proposed SVRG method, together with an initialization scheme and some different step size strategies. In Section \ref{sc:main result}, we establish the linear submanifold convergence of the proposed method. In Section \ref{sc:proof}, we present a key lemma as well as the proof for our main theorem. A series of experiments are provided in Section \ref{sc:experiment} to demonstrate the effectiveness of the proposed method as well as the effects of algorithmic parameters. We conclude this paper in Section \ref{sc:conclusion}. Part of proofs are presented in Appendix.

For any two matrices $X, Y \in \mathbb{R}^{p\times p}$, their inner product is defined as $\langle X, Y \rangle = \mathrm{tr}(X^TY)$. 
For any matrix $X \in \mathbb{R}^{p\times p}$, $\|X\|_F$ and $\|X\|_2$ denote its Frobenius and spectral norms, respectively, and $\sigma_{\min}(X)$ and $\sigma_{\max}(X)$ denote the smallest and largest \textit{strictly positive} singular values of $X$,  denote $\tau(X):= \frac{\sigma_{\max}(X)}{\sigma_{\min}(X)}$, with a slight abuse of notation, we also use $\sigma_1(X)\equiv \sigma_{\max}(X) \equiv \|X\|_2$.
${\bf I}_p$ denotes the identity matrix with the size $p\times p$. We will omit the subscript $p$ of ${\bf I}_p$ if there is no confusion in the context.

\section{A Stochastic Variance Reduced Gradient Scheme for SDP}
\label{sc:algorithms}

The SVRG method was firstly proposed by \cite{Johnson-Zhang-svrg2013} for minimizing a finite sum of convex functions with a vector argument.
The main idea of SVRG is adopting the variance reduction technique to accelerate SGD and achieves a faster convergence rate.
SVRG is an inner-outer loop based method.
The main purpose of inner loop is to reduce the variance introduced by the stochastic sampling, and thus accelerate the convergence of the outer loop iterates. In this section, we propose a new version of SVRG with Option I adapted to Semidefinie Programming which enjoys provable convergence guarantees. 

\subsection{SVRG-SDP with Option I: Algorithm \ref{alg:SVRG}}
Inspired by \cite{Ma2018,Ma2019} and \cite{Wang2017-ICML}, 
we propose a new variant of SVRG method with Option I to solve the stochastic SDP problem \eqref{Eq:SD-stoch} based on its nonconvex reformulation \eqref{Eq:SD-determ-ncvx}, as described in Algorithm \ref{alg:SVRG}.
Compared to SVRG-I suggested in \citep{Ma2018,Ma2019} for problem \eqref{Eq:SD-determ-ncvx}, a new semi-stochastic gradient 
\begin{equation}
    \nabla f_{i_t}(X^t)U^t - (\nabla f_{i_t}(\tilde{X}^k) - \nabla f(\tilde{X}^k))U^t
\end{equation}
is exploited in the inner loop to replace the original $\nabla f_{i_t}(X^t)U^t - (\nabla f_{i_t}(\tilde{X}^k)-\nabla f(\tilde{X}^k))\tilde{U}^k$ for the variance reduced estimate of the current gradient, where $\tilde{X}^k$ is the estimate at the $k$-th outer loop, $\tilde{U}^k$ is the associated factorization of $\tilde{X}^k$ (i.e., $\tilde{X}^k = \tilde{U}^k (\tilde{U}^k)^T$), and $U^t$ is the $t$-th iterate in the inner loop.
Intuitively, due to the use of the latest iterate $U^t$ in the inner loop, the new semi-stochastic gradient should be more accurate than the old one that mixed $U^t$ and $\tilde{U}^k$, resulting in a possible better performance than SVRG-I as demonstrated by our later numerical experiments. More importantly, the use of such a new semi-stochastic gradient is essential for deriving the linear convergence in the submanifold metric ${\cal E}(U,U_r^*) := \mathop{\min}_{R \in {\cal O}} \|U-U_r^*R\|_F^2$, as shown in the following convergence analysis.
Besides the Option I such that $\tilde{U}^{k+1} = U^m$, we also adopt the mini-batch strategy to SVRG-SDP, which (though obtainable) is missing in \citep{Wang2017-universal-SVRG,Zeng-svrg2018}.

When $m=1$ or $b=n$, SVRG-SDP reduces to the known factored gradient descent (FGD) method studied in \cite{Sanghavi-FGD2016}. This shows that our proposed algorithmic framework provides more flexible choices for the users.

\begin{algorithm}[t]
{\small
\begin{algorithmic}\caption{Stochastic Variance-Reduced Gradient for SDP problem (\textit{SVRG-SDP}) }\label{alg:SVRG}
\STATE {\bf Parameters}: an update frequency $m\geq 1$ of the inner loop, a mini-batch size $b\geq 1$, a sequence of positive step sizes $\{\eta_k\}$, an initial point $\tilde{U}^0 \in \mathbb{R}^{p \times r}$.
\smallskip
\FOR{$k=0,1,\ldots$}
\STATE $\tilde{X}^k := \tilde{U}^k {(\tilde{U}^k)}^T$
\STATE $U^0 = \tilde{U}^k$
\smallskip
\FOR{$t=0,\ldots,m-1$} 
\STATE $X^t = U^t {U^t}^T$
\STATE Pick ${\cal I}_t \subset \{1,\ldots,n\}$ with $|{\cal I}_t| = b$ via a uniformly random way
\STATE $U^{t+1} = U^t - \eta_k [\frac{1}{b}\sum_{i_t\in {\cal I}_t}(\nabla f_{i_t}(X^t) - \nabla f_{i_t}(\tilde{X}^k)) + \nabla f(\tilde{X}^k)]U^t$
\ENDFOR
\smallskip
\STATE $\tilde{U}^{k+1} = U^m$
\ENDFOR
\end{algorithmic}}
\end{algorithm}

\subsection{Initialization Procedure: Algorithm \ref{alg:initialization}}
Note that the recast problem \eqref{Eq:SD-determ-ncvx} is no longer convex, thus the choice of initialization is very important in the implementations of these algorithms for the low-rank matrix estimation as demonstrated in \citep{Sanghavi-FGD2016,Wang2017-universal-SVRG,Wang2017-GD,Zhang2017-SVRG,Ma2018,Ma2019,Zeng-svrg2018,Zeng-SGD2019}.
One of commonly used strategies is to construct the initialization directly from the observed data like in the applications of matrix sensing, matrix completion and phase retrieval
(say, \cite{Candes-TIT2015,Sanghavi2013,Netrapalli-2013,Zheng-Lafferty2015}).
Such a strategy is generally effective for the case that the objective function has a small ``condition number'', 
while for the general objective functions as considered in this paper, another common strategy is to use one of the standard convex algorithms (say, projected gradient descent (ProjGD) \citep{Nestrov-1989}).
Some specific implementations of this idea have been used in \citep{Recht-2016,Sanghavi-FGD2016,Wang2017-universal-SVRG,Wang2017-GD,Zeng-SGD2019}.

Motivated by the existing literature, this paper also suggests using the ProjGD method to generate the initialization $\tilde{U}^0$ for the general nonlinear loss function as described in Algorithm \ref{alg:initialization}, where $\mathrm{Proj}_{\mathbb{S}_+^{p}}$ represents the projection onto $\mathbb{S}_+^p$, the set of symmetric positive semidefinite matrices of size $p\times p$.

\begin{algorithm}[t]
{\small
\begin{algorithmic}\caption{Initialization for \textit{SVRG-SDP}}\label{alg:initialization}
\STATE Let $X^0 = 0$, $T$ be the number of iterations and $L$ be a Lipschit constant of $\nabla f$.
\smallskip
\FOR{$t=1,2,\ldots, T$}
\STATE ${X}^t = \mathrm{Proj}_{\mathbb{S}_+^{p}}({X}^{t-1}- \frac{1}{L}\nabla f({X}^{t-1}))$
\ENDFOR
\STATE Let $\tilde{X}^0 = X^T$, and $\tilde{U}^0$ be a rank-$r$ factorization of $\tilde{X}^0$.
\end{algorithmic}}
\end{algorithm}

\subsection{Fixed and Stabilized Barzilai-Borwein Step Sizes} 
Besides the issue of the initial choice,
another important implementation issue of SVRG is the tuning of the step size.
There are mainly two classes of step sizes: deterministic or data adaptive. Here we discuss three particular choices.
\begin{enumerate}
\item[(a)]
Fixed step size \citep{Johnson-Zhang-svrg2013}:
\begin{equation}
\label{Eq:fixed-stepsize}
\eta_k \equiv \eta, \quad \text{for some}\ \eta>0.
\end{equation}

\item[(b)]
Barzilai-Borwein (BB) step size \citep{BB-stepsize1988,Tan-Ma-bbsvrg2016}: given an initial $\eta_0 >0$ and for $k\geq 1$, let $\tilde{g}_k := \nabla f(\tilde{X}^k)$,
\begin{equation}
\label{Eq:bb-stepsize}
\eta_k = \frac{1}{m} \cdot \frac{\|\tilde{X}^k - \tilde{X}^{k-1}\|_F^2}{|\langle \tilde{X}^k - \tilde{X}^{k-1}, \tilde{g}_k - \tilde{g}_{k-1}\rangle|}.
\end{equation}
Note that such a BB step size is originally studied for strongly convex objective functions \citep{Tan-Ma-bbsvrg2016}, and it may be breakout if there is no guarantee of the curvature of $f$ like in nonconvex cases \citep{Ma2018}. In order to avoid such possible instability of \eqref{Eq:bb-stepsize} in our studies, a variant of BB step size, called the \textit{stabilized BB} step size, is suggested by \cite{Ma2018,Ma2019} shown as follows.

\item[(c)]
Stabilized BB (SBB) step size \citep{Ma2018,Ma2019}: given an initial $\eta_0 >0$ and an $\epsilon \geq 0$, for $k \geq 1$,
\begin{equation}
\label{Eq:sbb-stepsize}
\eta_k = \frac{1}{m} \cdot
\frac{\|\tilde{X}^k - \tilde{X}^{k-1}\|_F^2}{|\langle \tilde{X}^k - \tilde{X}^{k-1}, \tilde{g}_k - \tilde{g}_{k-1}\rangle| + \epsilon \|\tilde{X}^k - \tilde{X}^{k-1}\|_F^2}.
\end{equation}
Note that the BB step size is a special case of SBB step size with $\epsilon=0$, thus we call the BB step size as \text{SBB}$_0$.
\end{enumerate}

Besides the above three step sizes, there are some other schemes like the diminishing step size and the use of smoothing technique in BB step size as discussed in \cite{Tan-Ma-bbsvrg2016}.
However, we mainly focus on the listed two step sizes (as BB is a special case SBB$_0$) in this paper due to their established effectiveness in wide applications.

\section{Convergence Theory of SVRG-SDP}
\label{sc:main result}

In this section, we present the linear submanifold convergence of the proposed SVRG-SDP.

\subsection{Assumptions and Convergence Metric}

To present our main convergence results, we need the following assumptions.

\begin{assumption}
\label{Assump:objfun}
Suppose that
\begin{enumerate}
\item[(a)]
each $f_i$ ($i=1,\ldots,n$) is $L$-Lipschitz differentiable for some constant $L>0$, i.e., $f_i$ is smooth and $\nabla f_i$ is Lipschitz continuous satisfying
\[
\|\nabla f_i(X) - \nabla f_i(Y)\|_F \leq L \|X-Y\|_F, \ \forall X, Y \in \mathbb{S}_+^p.
\]

\item[(b)]
$f$ is $(\mu, r)$-restricted strongly convex (RSC) for some constants $\mu>0$ and $r \leq p$, i.e., for any $X, Y \in \mathbb{S}_+^p$ with rank $r$,
\[
f(Y) \geq f(X) + \langle \nabla f(X), Y-X\rangle + \frac{\mu}{2} \|Y-X\|_F^2.
\]
\end{enumerate}
\end{assumption}

The above assumptions are regular and commonly used in the literature (see, \cite{Sanghavi-FGD2016,Wang2017-universal-SVRG,Wang2017-ICML,Zeng-SGD2019}).
Assumption \ref{Assump:objfun}(a) implies that $f$ is also $L$-Lipschitz differentiable.
For any $L$-Lipschitz differentiable and $(\mu, r)$-restricted strongly convex function $h$, the following hold (\cite{Nestrov-2004}):
\begin{align*}
& h(Y) \leq h(X) + \langle \nabla h(X), Y-X \rangle + \frac{L}{2}\|Y-X\|_F^2, \\
& \mu \|X-Y\|_F^2 \leq \langle \nabla h(X) - \nabla h(Y), X-Y \rangle \leq L \|X-Y\|_F^2,
\end{align*}
where the first inequality holds for any $X, Y \in \mathbb{S}_+^p$, and the second inequality holds for any $X, Y \in \mathbb{S}_+^p$ with rank $r$, the first inequality and the right-hand side of the second inequality hold for the Lipschitz continuity of $\nabla h$, and the left-hand side of the second inequality is due to the $(\mu, r)$-restricted strong convexity of $h$.
Under Assumption \ref{Assump:objfun}, let
\begin{align}
& \kappa := \frac{L}{\mu}, \quad \gamma_0 := (\sqrt{2}-1)\kappa^{-1}, \label{Eq:kappa}
\end{align}
where $\kappa \geq 1$ is generally called the \textit{condition number} of the objective function.

In order to characterize the submanifold convergence of SVRG-SDP, we use the following orthogonally invariant metric to measure the gap between $U$ and $V$,
\[
{\cal E}(U,V) := \mathop{\min}_{R \in {\cal O}} \|U-VR\|_F^2, \quad \forall\  U, V \in \mathbb{R}^{p\times r},
\]
where ${\cal O}$ is the set of orthogonal matrices of size $r\times r$. Such metric has been widely used in the convergence analysis of low-rank factorization based algorithms in the literature (e.g., \cite{Sanghavi-FGD2016,Recht-2016,Wang2017-universal-SVRG,Wang2017-ICML,Zeng-SGD2019}).
Compared to the Euclidean metric $\tilde{\cal E}(U^t,U_r^*) := \|U^t-U_r^*\|_F^2$ used in the convergence analysis of SVRG-I in \cite{Zeng-svrg2018}, such an orthogonally invariant metric ${\cal E}(U,V)$ is more desired since a global minimum of the low-rank stochastic semidefinite programming problem \eqref{Eq:SD-stoch} is naturally invariant in the loss after an orthogonal transform on its factorizations.

\subsection{Local Linear Submanifold Convergence}

Let $X^*$ be a global optimum of problem \eqref{Eq:SD-stoch} with rank $r = \mathrm{rank}(X^*)$, and $U_r^* \in \mathbb{R}^{p\times r}$ be a rank-$r$ decomposition of $X^*$ (i.e., $X^* = U_r^* (U_r^*)^T$).
Let ${\cal N}_{\gamma_0}:= \{U:{\cal E}(U,U_r^*) \leq \gamma_0 \sigma_r(X^*)\}$, and we define the following constants:
\begin{align}
& B:= \max_{U \in {\cal N}_{\gamma_0}} \|UU^T\|_F, \label{Eq:B}\\
& \eta_{\max} := \frac{\gamma_0 \sigma_r(X^*)}{54(2+\sqrt{\gamma_0})^2L\|X^*\|_2B}.\label{Eq:eta-max}
\end{align}
Let $\{\eta_k\}_{k\in \mathbb{N}}$ be a sequence satisfying $\eta_k \in (0,b\eta_{\max}]$ for some mini-batch size $b$ satisfying
\begin{align}
\label{Eq:cond-b}
1\leq b \leq \min\left\{n,54(\sqrt{2}+1)\kappa \tau(X^*)\right\},
\end{align}
where $\tau(X^*) = \frac{\|X^*\|_2}{\sigma_r(X^*)}$.
Given a positive integer $m$, define
\begin{align}
\rho_k := \frac{1}{2}\left[1+\left(1-\frac{2}{27}\eta_kL\gamma_0\sigma_r(X^*) \right)^m \right]. \label{Eq:rho}
\end{align}
By \eqref{Eq:eta-max} and \eqref{Eq:cond-b}, there holds
\begin{align}
\label{Eq:cond-etak}
\eta_k \leq b\eta_{\max} \leq \frac{1}{4BL},
\end{align}
and thus,
\begin{align}
\label{Eq:cond-etak1}
\frac{2}{27}\eta_kL\gamma_0\sigma_r(X^*) \leq \frac{\sqrt{2}-1}{54 \kappa \tau(X^*)}<1,
\end{align}
where the second inequality holds for the definition of $B$ (implying $B \geq \|X^*\|_2$).
The above inequality shows that $\rho_k \in (0,1)$.
Based on the above defined constants, we present our main theorem on the local linear submanifold convergence of the proposed SVRG-SDP as follows, where its proof is postponed in Section \ref{sc:proof}.

\begin{theorem}[Local linear submanifold convergence]
\label{Thm:svrg}
Let $\{\tilde{U}^k\}$ be a sequence generated by Algorithm \ref{alg:SVRG}. Suppose that Assumption \ref{Assump:objfun} holds with $1<\kappa \leq 64(\sqrt{2}-1)$. Let $m\geq 1$, $b$ satisfy \eqref{Eq:cond-b}, and $\eta_k \in (0,b\eta_{\max}]$ for $k\in \mathbb{N}$. If the initialization $\tilde{U}^0$ satisfies $0<{\cal E}(\tilde{U}^0, U_r^*)< \frac{8(\sqrt{2}-1)\sigma_r(X^*)}{9\kappa}$, then for any positive integer $k$, there holds
\begin{align}
\label{Eq:linear-convergence}
\mathbb{E}[{\cal E}(\tilde{U}^{k},U_r^*)] \leq \left(\prod_{j=0}^{k-1} \rho_j\right) \cdot {\cal E}(\tilde{U}^0,U_r^*).
\end{align}
Particularly, if a fixed step size $\eta \in (0,b\eta_{\max}]$ is used, then the above inequality implies the following linear convergence,
\begin{align}
\label{Eq:local-linear-convergence}
\mathbb{E}[{\cal E}(\tilde{U}^{k},U_r^*)]
\leq \left(\frac{1}{2}\left[1+\left(1-\frac{2\eta L\gamma_0\sigma_r(X^*)}{27} \right)^m \right]\right)^k  {\cal E}(\tilde{U}^0,U_r^*).
\end{align}
\end{theorem}

Theorem \ref{Thm:svrg} establishes the local linear submanifold convergence of the proposed SVRG method for problem \eqref{Eq:SD-determ-ncvx} under the smoothness and RSC assumptions.
According to Theorem \ref{Thm:svrg}, the radius of the initial ball permitted to guarantee the linear convergence is $\frac{8(\sqrt{2}-1)\sigma_r(X^*)}{9\kappa}$, which is slightly better than the existing results as presented in Table \ref{Table:comparison-theory} and in some sense tight by \cite[Proposition 2]{Zeng-SGD2019},
where 
a counter example is provided such that FGD cannot converge to the global optimum once the initialization radius is not smaller than $\sigma_r(X_r^*)$.
In the following, we provide some detailed comparisons between them.

Under the same smoothness and RSC assumptions, similar local linear submanifold convergence of SVRG-LR was established in \cite[Theorem 4.7]{Wang2017-universal-SVRG} in the framework of both statistical and optimization frameworks. We provide some remarks on comparisons between these two results in the following.
At first, according to \cite[Theorem 4.7]{Wang2017-universal-SVRG} and the discussion in \cite[Remark 4.8]{Wang2017-universal-SVRG}, the provable radius of the initialization ball is $\frac{2\sigma_r(X^*)}{15\kappa}$ for SVRG-LR,
which is smaller than that required for the proposed SVRG-SDP in this paper. In this sense, the convergence conditions in this paper are weaker than those used in \cite{Wang2017-universal-SVRG}.
Secondly, from the discussion in \cite[Remark 4.8]{Wang2017-universal-SVRG}, the contraction parameter $\rho$ associated with SVRG-LR lies in the interval $[\frac{2}{3}, \frac{5}{6}]$ when $m=O(\tau^2(X^*))$, while by \eqref{Eq:rho}, the contraction parameter $\rho_k$ associated with our proposed SVRG-SDP (approaching to $1/2$) shall be smaller than $\frac{2}{3}$ if a moderately large $m$ is adopted.
This implies in some sense that the proposed version of SVRG as an Option I generally converges faster than SVRG-LR, the Option II counterpart for problem \eqref{Eq:SD-determ-ncvx}.
Moreover, note that for any positive integer $m$, $\rho_k$ defined in \eqref{Eq:rho} is always less than $1$.
This means we have no requirement on $m$ to guarantee the linear convergence of SVRG-SDP, while a sufficiently large $m$ (at least in the order of $O(\tau^2(X^*))$) is required to guarantee the linear convergence of SVRG-LR studied in \cite{Wang2017-universal-SVRG}.
Also, the convergence analysis of mini-batch version of SVRG-LR (though obtainable) is missing in the literature \citep{Wang2017-universal-SVRG}.

In Algorithm \ref{alg:SVRG}, when $m=1$, SVRG-SDP reduces to FGD studied in \cite{Sanghavi-FGD2016}.
Under the similar smoothness and RSC assumptions, the local linear convergence of FGD was firstly established in \cite{Sanghavi-FGD2016} if the initialization lies in the ball ${\cal E}(\tilde{U}^0,U_r^*)\leq \frac{\sigma_r(X^*)}{10000\kappa^2 \tau^2(X^*)}$, and later the radius of the initialization ball was improved to $\frac{(\sqrt{2}-1)\sigma_r(X^*)}{2\kappa}$ in \cite{Zeng-SGD2019}.
Notice that Theorem \ref{Thm:svrg} above also holds for the case of $m=1$.
This shows that the provable radius of initialization ball for FGD is further improved to $\frac{8(\sqrt{2}-1)\sigma_r(X^*)}{9\kappa}$.

Compared to the stochastic version of FGD, i.e., stochastic gradient descent (SGD) method studied in \cite{Zeng-SGD2019}, the convergence rate of SVRG-SDP is linear while that of SGD is sublinear when adopting a diminishing step size. As shown in Table \ref{Table:comparison-theory}, the initialization ball for SVRG-SDP is also slightly larger than that of SGD with a fixed step size in \cite{Zeng-SGD2019}, while in this case, SGD only converges to a $O(\eta)$-neighborhood of a global minimum, where $\eta$ is the used fixed step size.


Compared to SVRG-I studied in \cite{Zeng-svrg2018}, 
we establish the linear submanifold convergence, a more precise characterization in the low-rank matrix factorization setting, instead of point convergence for SVRG-SDP. 
The convergence conditions used in this paper are also weaker than those in \cite{Zeng-svrg2018} in the following two aspects. The first one is the weaker assumption on the objective function $f$. From \cite[Assumption 1(b)]{Zeng-svrg2018}, each $f_i$ is required to be $(\mu,r)$-restricted strongly convex; while in this paper, we only require the $(\mu,r)$-restricted strong convexity of the average function $f = \frac{1}{n}\sum_{i=1}^n f_i$ as shown in Assumption \ref{Assump:objfun}(b),
which is a more realistic condition used in the literature \citep{Sanghavi-FGD2016,Wang2017-universal-SVRG,Wang2017-GD,Wang2017-ICML,Zeng-SGD2019}.
The second one is that the initialization radius is improved from
$\frac{(2+\sqrt{3})\cdot(\sqrt{2}-1)\sigma_{r}(X^*)}{6\kappa}$ in \cite{Zeng-svrg2018} to
$\frac{8(\sqrt{2}-1)\sigma_r(X_r^*)}{9\kappa}$ in this paper, as shown in Table \ref{Table:comparison-theory}.
Moreover,  there are lack of theoretical guarantees on the initialization schemes (though obtainable) in \cite{Zeng-svrg2018}, while in this paper, we fill this gap and provide the theoretical guarantees for the suggested initialization scheme as shown in the following Proposition \ref{Propos:initial-projgrad}.


\begin{remark}[Influence of batch size and update frequency]
\label{remark:iteration-complexity}
Note that the range of the mini-batch size $b$ is very flexible,
since its upper bound is $54(\sqrt{2}+1)\kappa \tau(X^*)$, where $\kappa$ and $\tau(X^*)$ are generally far more than $1$.
If the particular fixed step size $\eta_k \equiv b \eta_{\max}$ is adopted in Algorithm \ref{alg:SVRG}, then \eqref{Eq:local-linear-convergence} implies
\begin{align}
\label{Eq:local-linear-convergence1}
\mathbb{E}[{\cal E}(\tilde{U}^{k},U_r^*)]
\leq \left(\frac{1}{2}\left[1+\left(1-\frac{2\eta_{\max} L\gamma_0\sigma_r(X^*)}{27} b \right)^m \right]\right)^k  {\cal E}(\tilde{U}^0,U_r^*).
\end{align}
Regardless of the memory storage, the above inequality shows that a larger mini-batch size $b$ adopted implies a faster convergence speed of the iterates yielded in the outer loop of SVRG-SDP, as long as $b$ is smaller than the upper bound specified in \eqref{Eq:cond-b}.
Similar claim also holds for the choice of update frequency $m$ in the inner loop by the above inequality since the base $1-\frac{2\eta_{\max} L\gamma_0\sigma_r(X^*)}{27} b$ is positive and less than $1$ under the choice of $b$ in \eqref{Eq:cond-b}. In order to yield a fast convergence speed, a moderately large $m$ is usually required in the implementation of SVRG-SDP, as also implied by \eqref{Eq:local-linear-convergence1}. They are reasonable since more gradient information is exploited when a large mini-batch size or update frequency is adopted in the inner loop with other fixed parameters.

Besides the concerned linear convergence of the iterates yielded in the outer loop, it is also important to estimate the computational complexity in terms of the amount of gradient information used to achieve a prescribed precision.
Specifically, given a precision $\epsilon>0$, the above inequality gives an estimate of the computational complexity of SVRG-SDP as follows:
\begin{align*}
{\cal C}(\epsilon,b,m) \geq \frac{b m }{C(b, m)}\cdot \log \frac{\epsilon}{{\cal E}(\tilde{U}^0,U_r^*)},
\end{align*}
where $C(b,m):=\log \frac{1}{2}\left[1+\left(1-\frac{2\eta_{\max} L\gamma_0\sigma_r(X^*)}{27} b \right)^m \right]$.
On one hand, for some moderately large $m$, the above computational complexity ${\cal C}(\epsilon,b,m)$ approximates to $\frac{bm}{\log 2}\log \frac{{\cal E}(\tilde{U}^0,U_r^*)}{\epsilon}$, which shows that a smaller mini-batch size generally implies a lower computational complexity of SVRG-SDP to achieve a prescribed precision.
On the other hand, for some fixed $b$ and considering some moderately large $m$, the above estimate of computational complexity also implies that a smaller $m$ leads to a lower computational complexity. These are also verified by our later numerical experiments in Section \ref{sc:experiment}.
\end{remark}

In the next, we give a corollary to show the linear convergence of SVRG-SDP when adopting the considered SBB step size (\ref{Eq:sbb-stepsize}), where BB step size is its special case with $\epsilon=0$.

\begin{corollary}[Convergence of SVRG-SDP with SBB step size]
\label{Coro:3stepsizes}
Suppose that the assumptions of Theorem \ref{Thm:svrg} hold and that $m \geq \frac{1}{(\mu+\epsilon) b\eta_{\max}}$ for any $\epsilon\geq 0$.
Then the convergence claim in Theorem \ref{Thm:svrg} also holds for SVRG-SDP equipped with the SBB step size in (\ref{Eq:sbb-stepsize}) .
%
\end{corollary}

By the definition of SBB step size \eqref{Eq:sbb-stepsize} and Assumption \ref{Assump:objfun}, we have
\[
\frac{1}{m(L+\epsilon)} \leq \eta_k \leq \frac{1}{m(\mu + \epsilon)}.
\]
Thus, if $m \geq \frac{1}{(\mu + \epsilon) b\eta_{\max}}$, then $\eta_k \leq  b\eta_{\max}$ for any $k \in \mathbb{N}$. This implies the convergence claim in Theorem \ref{Thm:svrg} holds for SVRG-SDP equipped with SBB step size.

\subsection{Global Linear Convergence with Provable Initial Scheme}





By Theorem \ref{Thm:svrg}, the radius of the initialization ball plays a central role in the establishment of local linear convergence of SVRG-SDP.
Thus, it is crucial to show whether such proper initial point can be easily achieved in practice.
In the next, we show that such a desired initial point can be indeed achieved by the suggested initialization scheme (see Algorithm \ref{alg:initialization}) with the logarithmic computational complexity.

\begin{proposition}
\label{Propos:initial-projgrad}
Let Assumption \ref{Assump:objfun} hold with $1< \kappa \leq 64(\sqrt{2}-1)$.
Let $\tilde{U}^0$ be yielded by Algorithm \ref{alg:initialization}.
If
\begin{align}
\label{Eq:T}
T \ge \log_{1-\kappa^{-1}} \frac{16(\sqrt{2}-1)^2 \sigma^2_r(X^*)}{9\kappa \|X^*\|_F^2},
\end{align}
then ${\cal E}(\tilde{U}^0,U_r^*) \leq \frac{8(\sqrt{2}-1) \sigma_r(X^*)}{9\kappa}.$
\end{proposition}

Similar results of Proposition \ref{Propos:initial-projgrad} have been shown in \cite{Wang2017-GD} for the gradient descent method and in \cite{Zeng-SGD2019} for the stochastic gradient descent method.
According to \cite[Theorem 5.7]{Wang2017-GD}, the condition on $\kappa$ is $\kappa \in (1,4/3)$,
while the requirement in Proposition \ref{Propos:initial-projgrad} is $\kappa \in (1,64(\sqrt{2}-1)]$,
which significantly relaxes the condition in \cite{Wang2017-GD}.
According to \cite[Proposition 1]{Zeng-SGD2019}, the requirements on $\kappa$ of both papers are the same.
However, the requirement on the radius of initialization ball in this paper is slightly weaker that of \cite{Zeng-SGD2019}, where the radius is improved from $\frac{(\sqrt{2}-1)\sigma_r(X^*)}{2\kappa}$  in \cite{Zeng-SGD2019} to $\frac{8(\sqrt{2}-1)\sigma_r(X^*)}{9\kappa}$ in this paper.
The proof of Proposition \ref{Propos:initial-projgrad} is provided in Appendix C.

Based on Proposition \ref{Propos:initial-projgrad}, we can boost the local linear convergence of SVRG-SDP shown in Theorem \ref{Thm:svrg} to the following global linear convergence.

\begin{theorem}[Global linear convergence]
\label{Thm:svrg-globalconv}
Suppose that Assumption \ref{Assump:objfun} holds with $1<\kappa \leq 64(\sqrt{2}-1)$.
Let $\{\tilde{U}^k\}$ be a sequence generated by Algorithm \ref{alg:SVRG} with $\tilde{U}^0$ via the initial scheme in Algorithm \ref{alg:initialization} (where $T$ satisfies \eqref{Eq:T}), $b$ satisfying \eqref{Eq:cond-b}, $m\geq 1$, and $\eta_k \in (0,b\eta_{\max}]$ for $k\in \mathbb{N}$.
Then $\{\tilde{U}^k\}$ converges to a global minimum exponentially fast.
\end{theorem}

The proof of this theorem is summarized as follows. By Theorem \ref{Thm:svrg}, we show that the proposed SVRG method converges to a global minimum exponentially fast starting from an initial guess close to this global minimum, and then according to Proposition \ref{Propos:initial-projgrad}, we show that the suggested initialization algorithm (see Algorithm \ref{alg:initialization}) can find the desired initial guess permitted to the linear convergence with an order of logarithmic computational complexity, starting from the trivial origin point. In other words, the convergence speed of the suggested initial scheme in Algorithm \ref{alg:initialization} is also linear to reach the desired initial precision. Therefore, combining Proposition \ref{Propos:initial-projgrad} with the local linear convergence in Theorem \ref{Thm:svrg}, the whole convergence speed of SVRG-SDP equipped with such initial scheme is linear starting from the origin point. This implies the global linear convergence of SVRG-SDP in Theorem \ref{Thm:svrg-globalconv}.

\section{Second-Order Descent Lemma and Proof of Theorem \ref{Thm:svrg}}
\label{sc:proof}

In this section, we present the key proofs of our main theorem (i.e., Theorem \ref{Thm:svrg}). The proof idea is motivated by \cite{Zeng-svrg2018} with an extension. Specifically, to prove Theorem \ref{Thm:svrg}, we need the following \textit{second-order descent} lemma which estimates the progress made by a single iterate of the inner loop and is characterized in the manifold metric instead of the Euclidean metric as in \cite[Lemma 1]{Zeng-svrg2018}. 

\begin{lemma}[Second-order descent lemma]
\label{Lemm:2nd-order-descent}
Let $\{U^t\}_{t=0}^{m}$ be the sequence at the $k$-th inner loop. Let Assumptions \ref{Assump:objfun} hold with $1<\kappa \leq 64(\sqrt{2}-1)$.
Let $b$ satisfy \eqref{Eq:cond-b}, and
let $\eta_k \in (0,b\eta_{\max})$ for any $k\in \mathbb{N}$, where $\eta_{\max}$ is specified in \eqref{Eq:eta-max}.
If $0<{\cal E}(\tilde{U}^k,U_r^*)< \gamma_0 \sigma_r(X^*)$ and $0<{\cal E}(U^t,U_r^*)< \gamma_0 \sigma_r(X^*)$ with $\gamma_0$ specified in \eqref{Eq:kappa}, then
\begin{align}
\mathbb{E}_{{\cal I}_t}[{\cal E}(U^{t+1},U_r^*)]
&\leq
\left( 1-\eta_k L \gamma_0 \sigma_r(X^*) + \eta_k^2 B_1\right) \cdot {\cal E}(U^t,U_r^*) \nonumber \\
&+ \eta_k L {\cal E}^2(U^t,U_r^*) +\eta_k^2 B_1 {\cal E}(\tilde{U}^k,U_r^*), \label{Eq:recur-innerloop}
\end{align}
where $B_1 = \frac{2}{b}(2+\sqrt{\gamma_0})^2L^2 \|X^*\|_2 B$.
\end{lemma}

We call this lemma as \textit{Second-order descent lemma} since both linear and quadratic terms of ${\cal E}(U^t,U_r^*)$ are involved in the upper bound in the right-hand side of \eqref{Eq:recur-innerloop}. This is in general different from the literature (say, \cite{Johnson-Zhang-svrg2013,Tan-Ma-bbsvrg2016}) to yield the linear convergence of SVRG methods.

\medskip

\begin{proof}[Proof of Lemma \ref{Lemm:2nd-order-descent}]
Let $v_k^t:= \frac{1}{b}\sum_{i_t \in {\cal I}_t}(\nabla f_{i_t}(X^t) - \nabla f_{i_t}(\tilde{X}^k)) + \nabla f(\tilde{X}^k)$,
$R_{U^t}^* = \arg \min_{R\in {\cal O}} \|{U^t}-U_r^*R\|_F^2$, and $V_{U^t}^*: = U_r^* R_{U^t}^*.$
Note that
\begin{align*}
& {\cal E}(U^{t+1},U_r^*)
 = \min_{R\in {\cal O}} \|U^{t+1} - U_r^*R\|_F^2\\
& \leq \|U^{t+1}-V_{U^t}^*\|_F^2\\
& = \|U^{t}-V_{U^t}^*\|_F^2 + 2\langle U^{t+1}-U^t, U^t-V_{U^t}^* \rangle + \|U^{t+1}-U^t\|_F^2\\
& = {\cal E}(U^t,U_r^*) - 2\eta_k \langle v_k^tU^t, U^t - V_{U^t}^*\rangle + \eta_k^2 \|v_k^tU^t\|_F^2,
\end{align*}
which implies
\begin{align}
{\mathbb E}_{{\cal I}_t}[{\cal E}(U^{t+1},U_r^*)]
\leq {\cal E}(U^t,U_r^*) - 2\eta_k {\mathbb E}_{{\cal I}_t}[\langle v_k^tU^t, U^t - V_{U^t}^*\rangle]
+ \eta_k^2 {\mathbb E}_{{\cal I}_t}[\|v_k^tU^t\|_F^2]. \label{Eq:keyineq1}
\end{align}

The bounds of both $2{\mathbb E}_{{\cal I}_t}[\langle v_k^tU^t, U^t - V_{U^t}^*\rangle]$ and ${\mathbb E}_{{\cal I}_t}[\|v_k^tU^t\|_F^2]$ can be estimated respectively by Lemma \ref{Lemm:bound-innerproduct} and Lema \ref{Lemm:bound-variance} whose proofs are given in Appendix B.
Combining these bounds ((\ref{Eq:T1-lowerbound}) and (\ref{Eq:bound-variance}) in Appendix B) for (\ref{Eq:keyineq1}) yields
\begin{align*}
&{\mathbb E}_{{\cal I}_t}[{\cal E}(U^{t+1},U_r^*)]\\
&\leq {\cal E}(U^t,U_r^*) + \eta_k L {\cal E}^2(U^t,U_r^*)
-\eta_k (\sqrt{2}-1)\mu \sigma_r(X^*) {\cal E}(U^t,U_r^*) - \eta_k(\frac{1}{4L}-\eta_k B) \|{\cal P}_{U^t}\nabla f(X^t)\|_F^2 \\
&+ \eta_k^2{ \frac{2}{b}}(2+\sqrt{\gamma_0})^2 L^2 \|X^*\|_2 B {\cal E}(U^t,U_r^*)
+\eta_k^2{ \frac{2}{b}}(2+\sqrt{\gamma_0})^2 L^2 \|X_r^*\|_2 B {\cal E}(\tilde{U}^k, U_r^*)
\\
&\leq {\cal E}(U^t,U_r^*)  -\eta_k \left( (\sqrt{2}-1)\mu\sigma_r(X^*)-\eta_kB_1\right){\cal E}(U^t,U_r^*)
+\eta_k^2B_1 {\cal E}(\tilde{U}^k,U_r^*) + \eta_k L {\cal E}^2(U^t,U_r^*)\\
&= \left( 1-\eta_k L \gamma_0 \sigma_r(X_r^*) + \eta_k^2 B_1\right) \cdot {\cal E}(U^t,U_r^*)
+ \eta_k L {\cal E}^2(U^t,U_r^*) +\eta_k^2 B_1 {\cal E}(\tilde{U}^k,U_r^*) ,
\end{align*}
where $B_1 := {\frac{2}{b}}(2+\sqrt{\gamma_0})^2L^2 \|X_r^*\|_2 B$, the second inequality holds for $\eta_k \leq \frac{1}{4BL}$
by \eqref{Eq:cond-etak},
and
and the final equality holds for $\gamma_0 = \frac{(\sqrt{2}-1)\mu}{L}$.
Thus, the above inequality yields \eqref{Eq:recur-innerloop}.
\end{proof}

Equipped with Lemma \ref{Lemm:2nd-order-descent}, we are able to present the proof of Theorem \ref{Thm:svrg} as follows.
\medskip

\begin{proof}[Proof of Theorem \ref{Thm:svrg}]
We prove this theorem by induction.
We firstly develop the \textit{contraction} between two iterates of the outer loop based on Lemma \ref{Lemm:2nd-order-descent}, and then establish the locally linear convergence recursively.

We assume that at the $k$-th inner loop, ${\cal E}(\tilde{U}^k,U_r^*)<\frac{8(\sqrt{2}-1)\sigma_r(X_r^*)}{9\kappa}$ and ${\cal E}({U}^t,U_r^*)<\frac{8(\sqrt{2}-1)\sigma_r(X_r^*)}{9\kappa}$ for $k\geq 1$ and $1\leq t\leq m$,
then \eqref{Eq:recur-innerloop} still holds due to $\frac{8(\sqrt{2}-1)\sigma_r(X_r^*)}{9\kappa}=\frac{8}{9}\gamma_0 \sigma_r(X^*) < \gamma_0 \sigma_r(X_r^*)$. This implies
\begin{align*}
&\mathbb{E}_{{\cal I}_t}[{\cal E}(U^{t+1},U_r^*)]\\
&\leq \left(1-\frac{1}{9}\eta_k L \gamma_0 \sigma_r(X^*) + \eta_k^2B_1 \right) \cdot {\cal E}(U^t,U_r^*)
+ \eta_k^2 B_1 {\cal E}(\tilde{U}^k,U_r^*)\nonumber\\
&\leq \left(1-\frac{1}{9}\eta_k L \gamma_0 \sigma_r(X^*) + \eta_k^2B_1 \right) \cdot {\cal E}(U^t,U_r^*)
+ \frac{1}{2}\left( \frac{1}{9}\eta_k L \gamma_0 \sigma_r(X^*) - \eta_k^2B_1 \right) {\cal E}(\tilde{U}^k,U_r^*)
\end{align*}
where the first inequality holds for ${\cal E}({U}^t,U_r^*)<\frac{8}{9}\gamma_0\sigma_r(X_r^*)$,
and the second inequality holds for $\eta_k \leq \frac{b \gamma_0\sigma_r(X_r^*)}{54(2+\sqrt{\gamma_0})^2 \|X^*\|_2 LB}$ (implying $ \eta_k^2 B_1 \leq \frac{1}{27}\eta_kL\gamma_0\sigma_r(X^*)$).
Adding the term $-\frac{1}{2}{\cal E}(\tilde{U}^k,U_r^*)$ to both sides of the above inequality yields
\begin{align}
&\mathbb{E}_{{\cal I}_t}[{\cal E}(U^{t+1},U_r^*)] -\frac{1}{2}{\cal E}(\tilde{U}^k,U_r^*) \nonumber\\
&\leq \left(1-\frac{1}{9}\eta_k L \gamma_0 \sigma_r(X^*) + \eta_k^2B_1 \right) \left({\cal E}(U^{t},U_r^*) -\frac{1}{2}{\cal E}(\tilde{U}^k,U_r^*)\right) \nonumber\\
&\leq \left(1-\frac{2}{27}\eta_k L \gamma_0 \sigma_r(X^*) \right) \left({\cal E}(U^{t},U_r^*)-\frac{1}{2}{\cal E}(\tilde{U}^k,U_r^*)\right),
\label{Eq:contract-innerloop}
\end{align}
where the second inequality holds by $\eta_k^2 B_1 \leq \frac{1}{27}\eta_kL\gamma_0\sigma_r(X^*)$ again.
Note that for any $k\in \mathbb{N}$, $1-\frac{2}{27}\eta_k L \gamma_0 \sigma_r(X^*)\in (0,1)$ by \eqref{Eq:cond-etak1} and $U^0 = \tilde{U}^k$ at the $k$-th inner loop, then based on \eqref{Eq:contract-innerloop}, we have
\begin{align*}
\mathbb{E}[{\cal E}(\tilde{U}^{k+1},U_r^*)]
\leq \frac{1}{2}\left( 1+ \left(1-\frac{2}{27}\eta_k L \gamma_0 \sigma_r(X^*) \right)^m\right){\cal E}(\tilde{U}^k,U_r^*),
\end{align*}
which implies for any positive integer $k$
\begin{align*}
\mathbb{E}[{\cal E}(\tilde{U}^{k},U_r^*)] \leq \left(\prod_{j=0}^{k-1} \rho_j\right) \cdot {\cal E}(\tilde{U}^0,U_r^*),
\end{align*}
where $\rho_j:=\frac{1}{2}\left( 1+ \left(1-\frac{2}{27}\eta_j L \gamma_0 \sigma_r(X^*) \right)^m\right)$.
This finishes the proof of the theorem.
\end{proof}

\section{Numerical Experiments}
\label{sc:experiment}

In this section, we conduct a series of experiments in matrix sensing to show the effectiveness of the proposed method as well as the effects of the associated parametric parameters including the choices of step size,  mini-batch size and update frequency involved in the inner loop.

\subsection{Experimental settings}
\label{sc:experimental-settings}
The following matrix sensing problem
\begin{equation*}
\min_{X\succeq 0} f(X):= \frac{1}{2n} \sum_{i=1}^n (y_i - \langle A_i,X \rangle)^2,
\end{equation*}
is considered in this paper, where $X \in \mathbb{S}^p$ is a low-rank matrix with rank $r<p$, $A_i \in \mathbb{R}^{p\times p}$ is a sub-Gaussian independent measurement matrix of the $i$-th sample, $y_i \in \mathbb{R}$, and $n\in \mathbb{N}$ is the sample size.

Specifically, in these experiments, we let $p=100$, the optimal matrix $X^* = U^*(U^*)^T$ be a low-rank matrix with rank $5$ and the sample size $n=10p$.
In order to demonstrate the effectiveness of the proposed method, we consider the following batch and stochastic methods as competitors: factored gradient descent (FGD) method \citep{Sanghavi2013}, stochastic gradient descent (SGD) methods \citep{Zeng-SGD2019} with a fixed step size (SGD-fix) or diminishing step sizes (SGD-diminish), SVRG with with Option I (SVRG-I, where the update in the inner loop of Algorithm \ref{alg:SVRG} is the traditional one, $\nabla f_{i_t}(X^t)U^t - (\nabla f_{i_t}(\tilde{X}^k)-\nabla f(\tilde{X}^k))\tilde{U}^k$), and SVRG-LR \citep{Wang2017-universal-SVRG} which is of Option II.
For all algorithms, we take $r=r^*$ and construct the initialization empirically via $10$ iterations of the projected gradient descent as similar to \cite[Section V.A]{Zeng-SGD2019}. For FGD, we use a step size $\eta = \frac{1}{4L\|X^*\|_F}$, where $L$ is the Lipschitz constant of $\nabla f(X)$.
For SGD, we set a more consecutive fixed step size $\bar{\eta} = \frac{1}{8L\|X^*\|_F}$ and a diminishing step size $\eta_k = \frac{\bar{\eta}}{k+1}$.
For SVRG type of methods, we set $b=1$, $m=n$ and $\eta= \frac{1}{4L\|X^*\|_F}$. All these parametric settings are set either theoretically or empirically optimal.

There are four experiments conducted in this section. The first three experiments are conducted to show the effects of algorithmic parameters including the choice of mini-batch size, update frequency and step size strategy involved in the inner loop, while the last experiment is implemented to demonstrate the effectiveness of the proposed method via comparing with these batch and stochastic methods presented in the previous.

In all experiments, we use the \textit{relative recovery error} defined by $\frac{\|\tilde{X}^k-X^*\|_F^2}{\|X^*\|_F^2}$ as the evaluation metric, where $\tilde{X}^k$ is the estimate yielded at the $k$-th iteration of outer loop as presented in Algorithm \ref{alg:SVRG}.

\subsection{Effect of mini-batch size}

In this experiment, we study the effect of mini-batch size. Under the same settings as described in Section \ref{sc:experimental-settings}, we particularly consider five choices of mini-batch sizes, i.e., $\{1, 2, 5, 8, 10, 20\}$ for the proposed method. We set the update frequency $m=n$ for all five choices of mini-batch size.
The experiment results are shown in Figure \ref{Fig:mini-batch-size}.
From Figure \ref{Fig:mini-batch-size}(a), the proposed method exhibits the linear convergence for all the choices of mini-batch size, which verifies the established theoretical results in Theorem \ref{Thm:svrg} and Theorem \ref{Thm:svrg-globalconv}.
Particularly, concerning the convergence speed of iterates yielded in the outer loop of SVRG-SDP, a larger mini-batch size generally implies a faster convergence speed. This verifies our theoretical discussion in Remark \ref{remark:iteration-complexity}.  Moreover, we also show the effect of the mini-batch size to the computational complexity of the proposed method in terms of running time. As demonstrated by Figure \ref{Fig:mini-batch-size}(b), a smaller mini-batch size leads to lower computational complexity to achieve the same recovery precision. This also coincides with our computational complexity analysis as presented in Remark \ref{remark:iteration-complexity}. Motivated by this experiment, we set the mini-batch size of SVRG-SDP to be 1 in the following experiments.

\begin{figure}[!t]
\begin{minipage}[b]{0.49\linewidth}
\centering
\includegraphics*[scale=.48]{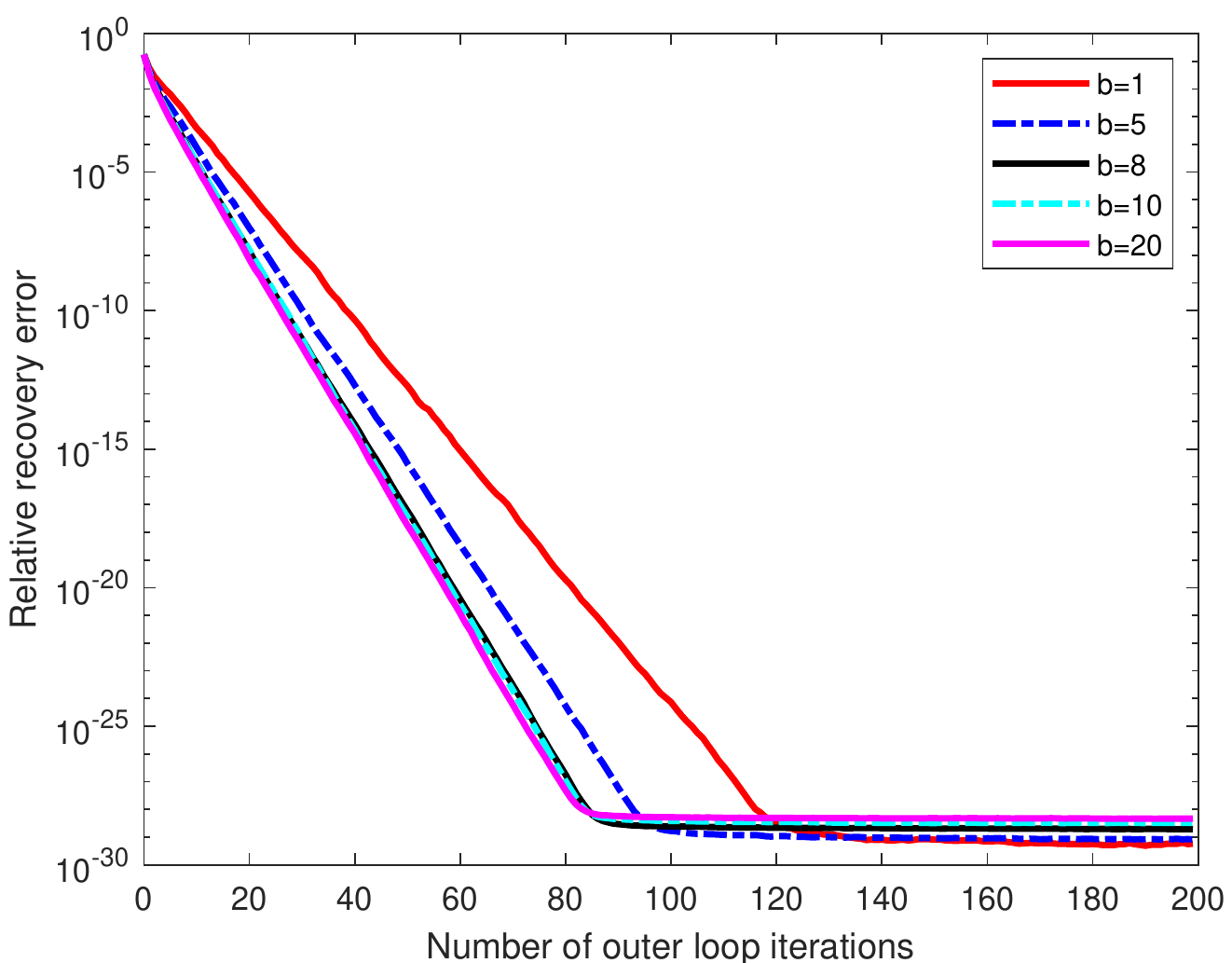}
\centerline{{\small (a) Convergence speed}}
\end{minipage}
\hfill
\begin{minipage}[b]{0.49\linewidth}
\centering
\includegraphics*[scale=.48]{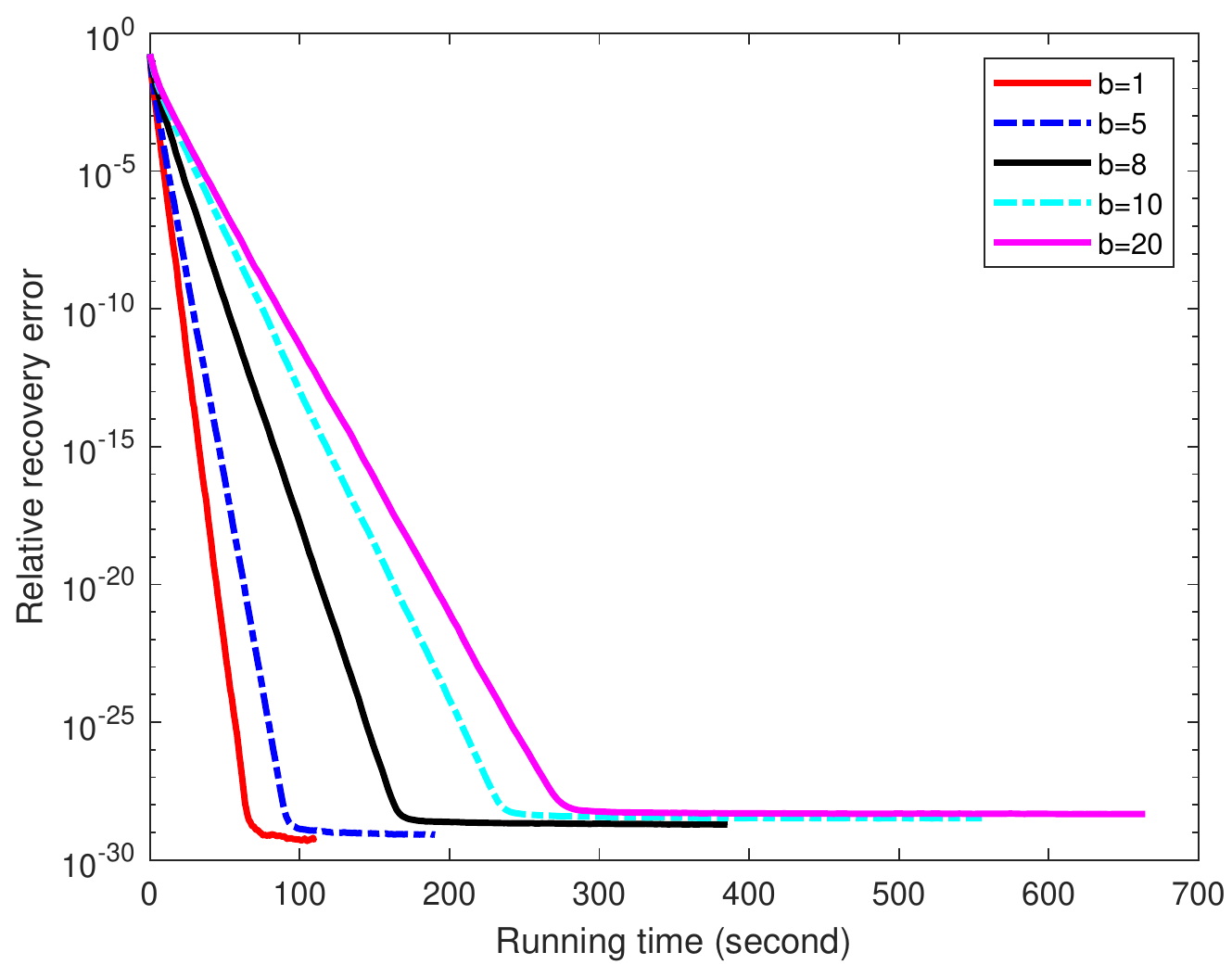}
\centerline{{\small (b) Computational complexity}}
\end{minipage}
\hfill
\caption{Effect of mini-batch size for SVRG-SDP. (a) Convergence speed of iterates yielded in the outer loop of SVRG-SDP, (b) Recovery error with respect to running time.
}
\label{Fig:mini-batch-size}
\end{figure}

\subsection{Effect of update frequency}
We conduct this experiment to study the effect of update frequency under the same settings described in Section \ref{sc:experimental-settings}. Specifically, we consider five choices of update frequency, i.e., $m\in \{\frac{n}{8},\frac{n}{4},\frac{n}{2},n,2n\}$, where $n$ is the sample size. The effects of the choice of update frequency to the convergence speed of iterates yielded in the outer loop of SVRG-SDP and computational complexity are presented in Figure \ref{Fig:update-frequency}. From Figure \ref{Fig:update-frequency}(a), a larger $m$ implies a faster convergence speed of the iterates yielded in the outer loop of SVRG-SDP, which coincides with the analysis in Remark \ref{remark:iteration-complexity}.
From Figure \ref{Fig:update-frequency}(b), the performance of SVRG-SDP with $m=n$ is the best among all these five choices in the sense that the computational complexity of SVRG-SDP with $m=n$ is the lowest to achieve a given recovery precision, where the performance of SVRG-SDP with $m=n/2$ is very close to that of SVRG-SDP with $m=n$. This also coincides with our analysis in Remark \ref{remark:iteration-complexity} and the empirical study for SVRG in the literature \cite{Ma2019}. Thus, in the later experiments, we fix $m=n$ in the implementation of SVRG-SDP.

\begin{figure}[!t]
\begin{minipage}[b]{0.49\linewidth}
\centering
\includegraphics*[scale=.48]{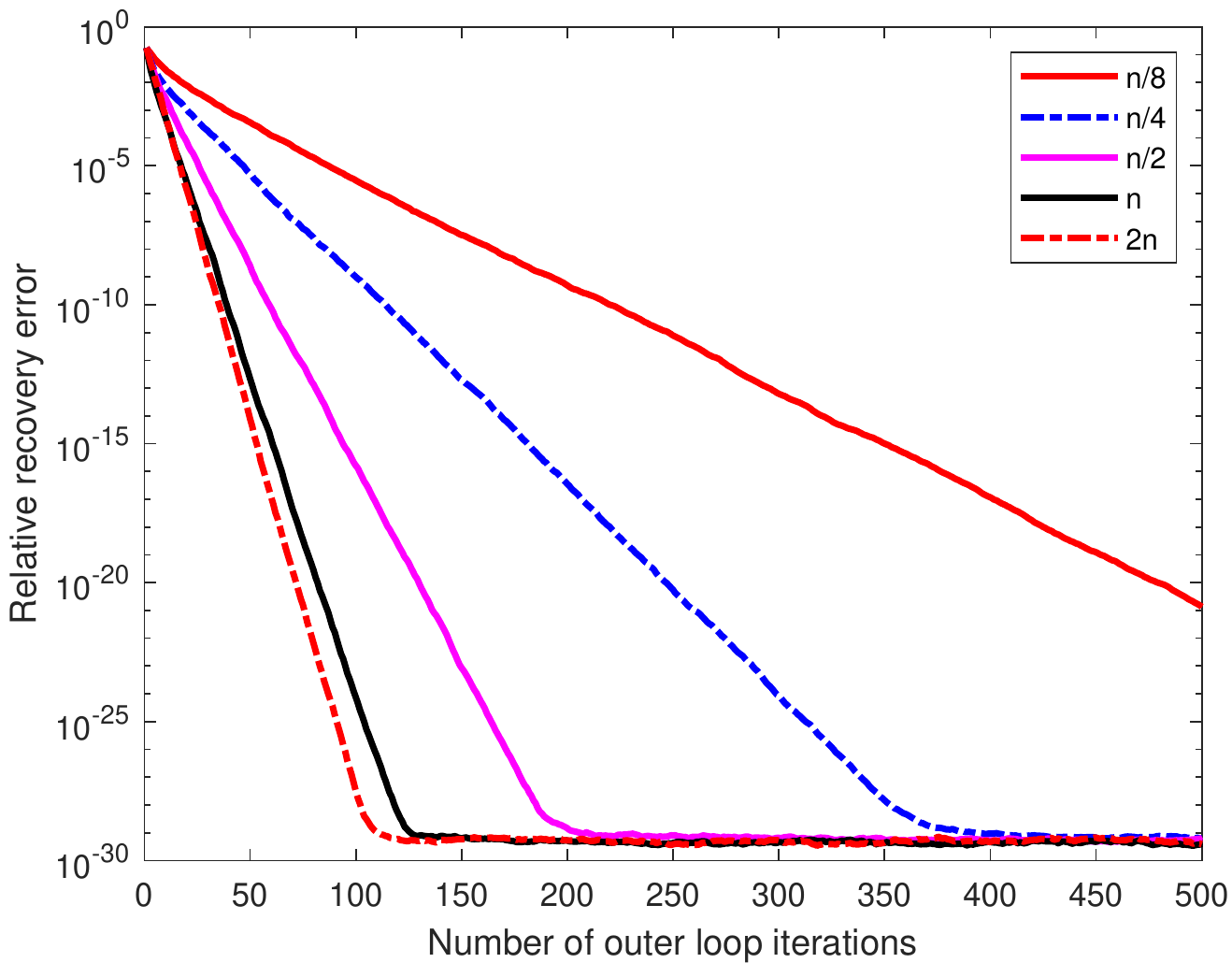}
\centerline{{\small (a) Convergence speed}}
\end{minipage}
\hfill
\begin{minipage}[b]{0.49\linewidth}
\centering
\includegraphics*[scale=.48]{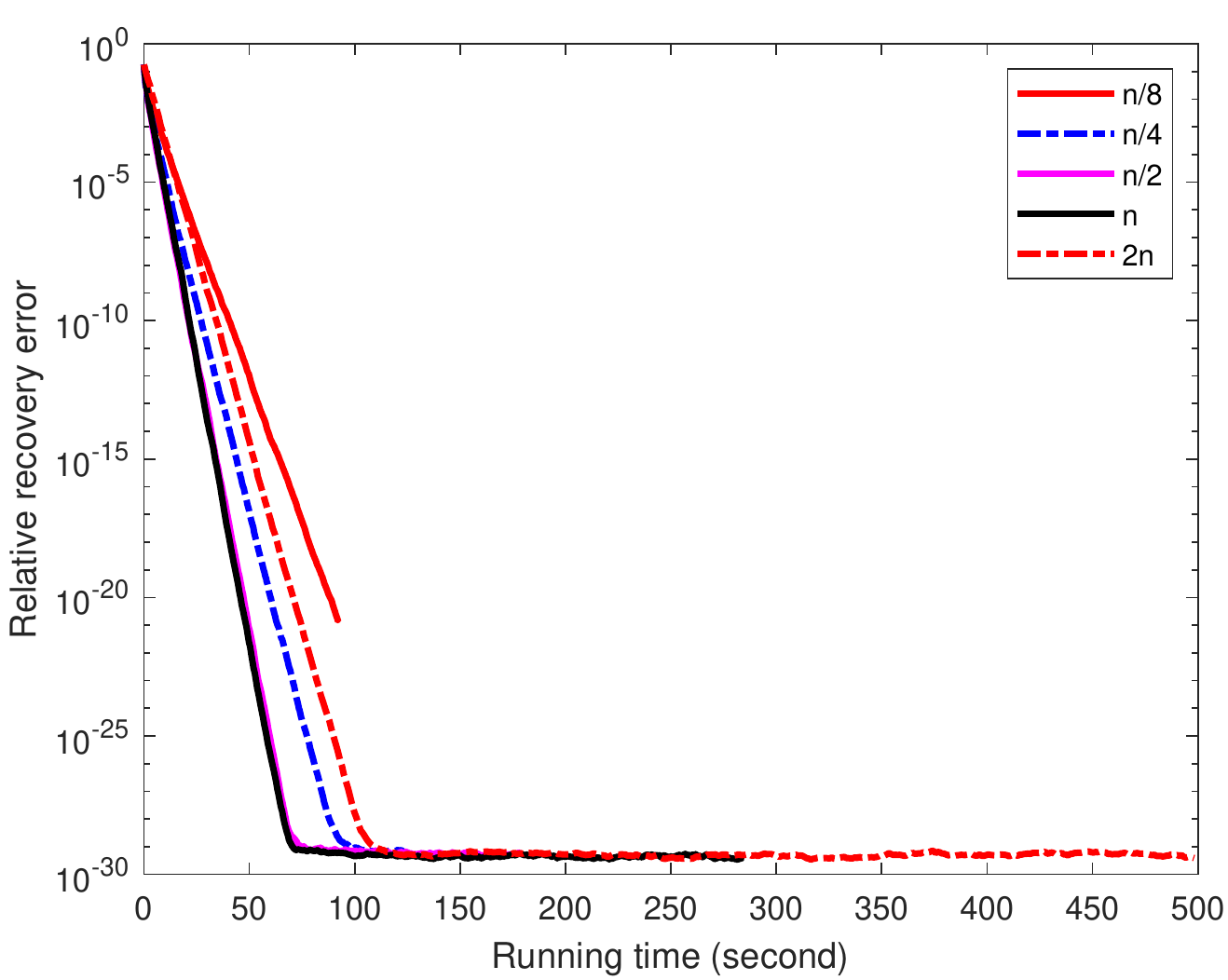}
\centerline{{\small (b) Computational complexity}}
\end{minipage}
\hfill
\caption{Effect of update frequency of the inner loop for SVRG-SDP. (a) Convergence speed of iterates yielded in the outer loop of SVRG-SDP, (b) Recovery error with respect to running time.
}
\label{Fig:update-frequency}
\end{figure}

\subsection{Performance of different step sizes}

We compare the performance of SVRG-SDP with two different types of step sizes, i.e., fixed and SBB step sizes under the same experimental settings as described in Section \ref{sc:experimental-settings}. Specifically, we consider a fixed step size $\eta= \frac{1}{4L\|X^*\|_F}$ and three SBB step sizes with $\epsilon = 0, 10^{-10}, 10^{-5}$, where SBB with $\epsilon=0$ is exactly the BB step size \citep{BB-stepsize1988}. For SVRG-SDP, we set $m=n$ and $b=1$ for all schemes of step sizes, motivated by the previous experiments. The experiment result is shown in Figure \ref{Fig:step-size}, where the epoch number in the horizontal axis is defined as the number of rounds that the total $n$ gradients are used. Since $m=n$ and $b=1$ in this experiment, an epoch of SVRG-SDP only includes one iteration of the outer loop. By Figure \ref{Fig:step-size}, the performance of SVRG-SDP equipped with the fixed step size is slightly better than concerned SVRG-SDP with both SBB and BB step sizes. When concerning the performance of SBB step size with a small $\epsilon$, the parameter $\epsilon$ has little effect on the performance of the proposed method.

\begin{figure}[!t]
\begin{minipage}[b]{0.99\linewidth}
\centering
\includegraphics*[scale=.6]{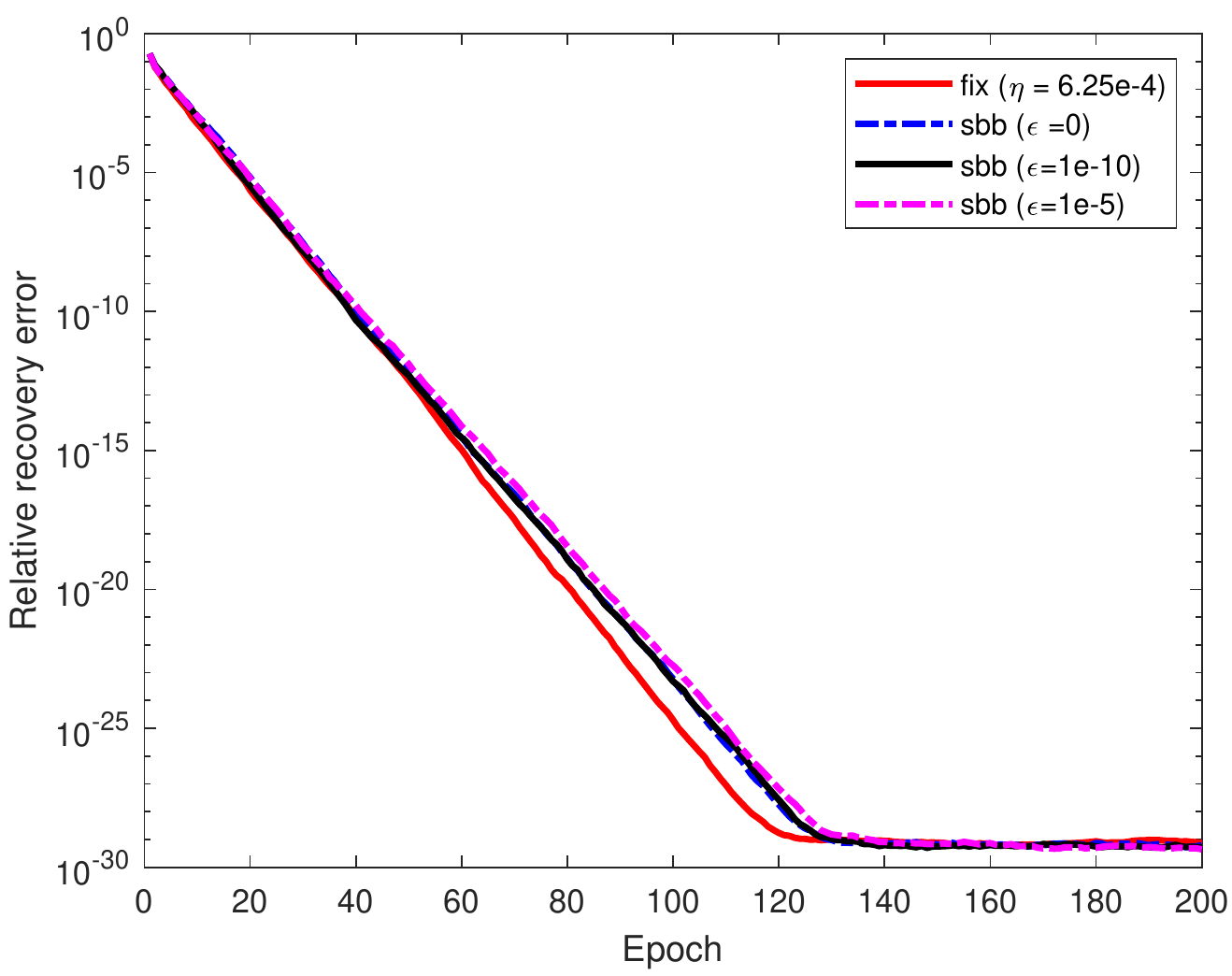}
\end{minipage}
\hfill
\caption{Performance of SVRG-SDP with different step sizes. 
}
\label{Fig:step-size}
\end{figure}

\subsection{Comparison with other algorithms}
\label{sc:comp-alg}

In this experiment, we demonstrate the effectiveness of SVRG-SDP via comparing to the other bath and stochastic methods including FGD \citep{Sanghavi2013}, SGD-fix \citep{Zeng-SGD2019}, SGD-diminish, SVRG-I \citep{Ma2018,Ma2019} and SVRG-LR \citep{Wang2017-universal-SVRG}. The experiment settings are presented in Section \ref{sc:experimental-settings}. The performance of these algorithms is shown in Figure \ref{Fig:comp-alg}. From Figure \ref{Fig:comp-alg}, all these SVRG type methods have the linear convergence and significantly outperform FGD and SGD-diminish. This justifies the established theoretical results for SVRG methods in the literature and this paper.
In particular, the proposed SVRG-SDP outperforms the existing batch and stochastic methods, where the performance of SVRG-I is very close or slightly inferior to that of SVRG-SDP. However, the convergence of SVRG-I is not known guaranteed. 
Moreover, when compared to SVRG-LR, the Option II counterpart of SVRG-SDP proposed in \cite{Wang2017-universal-SVRG}, our proposed SVRG-SDP works much better than SVRG-LR in terms of the convergence speed.
These show the advantage of Option I scheme and particularly the effectiveness of our proposed method.

\begin{figure}[!t]
\begin{minipage}[b]{0.99\linewidth}
\centering
\includegraphics*[scale=.6]{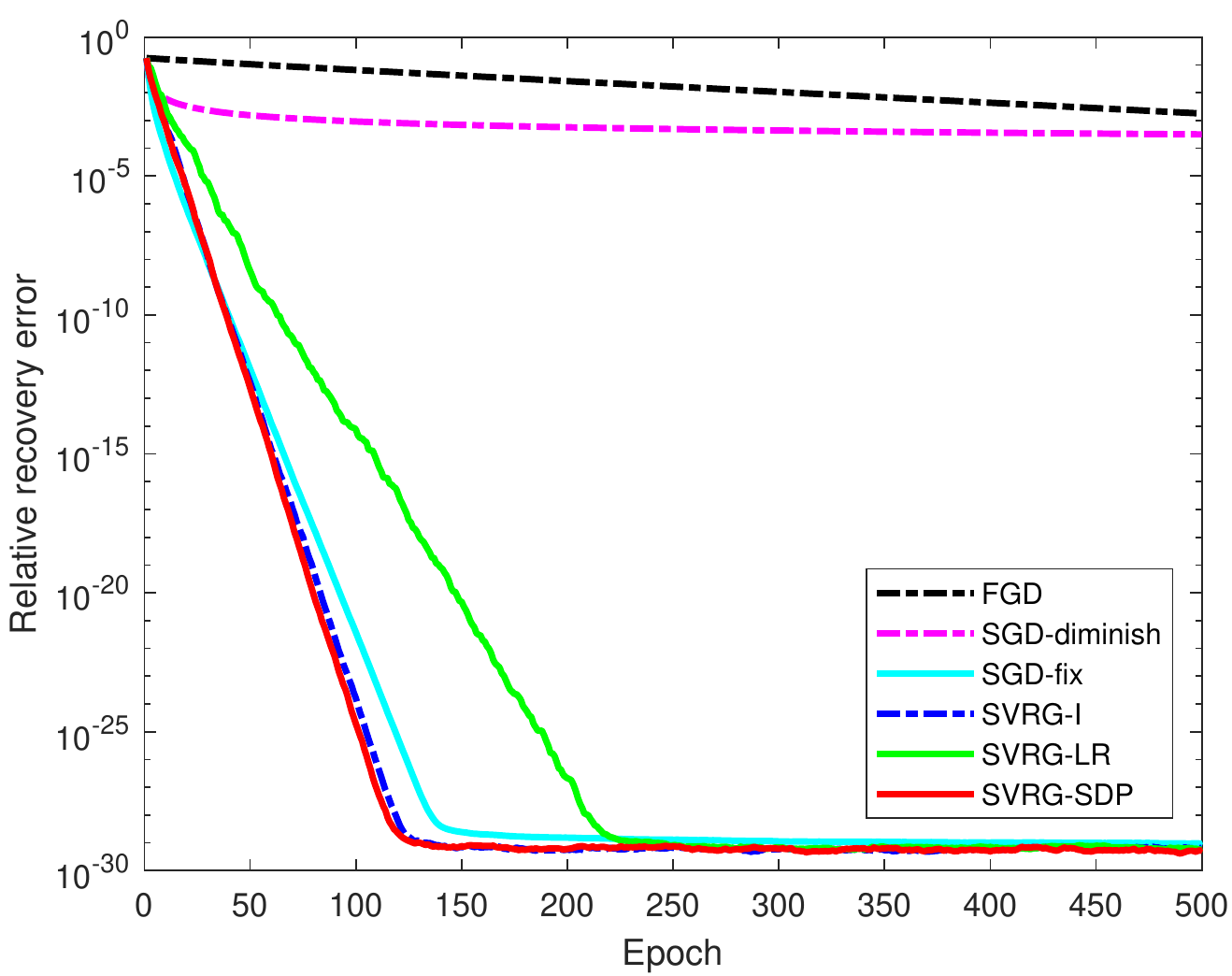}
\end{minipage}
\hfill
\caption{Comparison on the performance of different algorithms.
}
\label{Fig:comp-alg}
\end{figure}

\section{Conclusion}
\label{sc:conclusion}
In this paper, we propose an efficient SVRG method with Option I (i.e., the output of inner loop is chosen as the last iterate) called \textit{SVRG-SDP} for the stochastic semidefinite optimization problem involved in many applications, via exploiting a new semi-stochastic gradient in the update of inner loop.
The local linear submanifold convergence of the proposed version of SVRG method is established under the regular smoothness and RSC assumptions of the objective function at a proper initial choice. The provable radius of the initialization ball yielded in this paper is larger than those in the existing literature.
The local linear convergence of the proposed method can be boosted to the global linear convergence when adopting an efficient and provable initialization scheme.
Moreover, several step size schemes including the fixed, BB and SBB step sizes are included in the proposed method with provable guarantees.
To demonstrate the effectiveness of the proposed method, a series of experiments in matrix sensing are conducted to compare both batch and stochastic algorithms, where the effects of algorithmic parameters include the mini-batch size, update frequency and step size scheme studied theoretically and numerically.
The proposed SVRG method and its associated theoretical results established in this paper fill in an existing gap between the theory and practice of SVRG methods. 

\appendix
\section*{Appendix A. Preliminary Lemmas}

We provide some preliminary lemmas, which are frequently used in the later proofs.
All these lemmas are either directly taken or slightly adapted from the literature.

\begin{lemma}[Lemma 5.4 in \cite{Recht-2016}]
\label{Lemm:dist-twomat}
For any $U,V \in \mathbb{R}^{p\times r},$  then we have
\[
\|UU^T - VV^T\|_F^2 \geq 2(\sqrt{2}-1)\sigma_r^2(V)\cdot {\cal E}(U,V).
\]
\end{lemma}

\begin{lemma}[Lemma 3 in \cite{Zeng-SGD2019}]
\label{Lemm:sigma_min_X}
For any $U\in \mathbb{R}^{p\times r}$, let $X = UU^T$. If ${\cal E}(U,U_r^*) \leq \gamma \sigma_r(X^*)$ for some constant $0<\gamma <1$, then
\[
\sigma_r(X) \geq (1-\sqrt{\gamma})^2 \sigma_r(X^*).
\]
\end{lemma}

\begin{lemma}[Lemma 4 in \cite{Zeng-SGD2019}]
\label{Lemm:X-Xr*}
For any $U\in \mathbb{R}^{p\times r}$, let $X = UU^T$. If ${\cal E}(U,U_r^*) \leq \gamma \sigma_r(X^*)$ for some constant $0<\gamma<1$, then
\begin{align*}
\|X-X^*\|_F
\leq (2+\sqrt{\gamma})\|U_r^*\|_2 {\cal E}^{1/2}(U,U_r^*)
\leq (2\sqrt{\gamma}+\gamma)\cdot \tau(U_r^*)\cdot \sigma_r(X^*),
\end{align*}
where $\tau(U_r^*):=\frac{\sigma_1(U_r^*)}{\sigma_r(U_r^*)}$.
\end{lemma}

For any matrix $U \in \mathbb{R}^{p\times r}$, let $Q_U$ be a basis of the column space of $U$. Denote ${\cal P}_U := Q_UQ_U^T$. Then ${\cal P}_U  U = U$. For any matrix $Y \in \mathbb{R}^{p\times p}$, ${\cal P}_UY$ is a projection of $Y$ onto the same subspace spanned by $X:= UU^T$.
The following lemma is a special case of \cite[Lemma 5]{Zeng-SGD2019} with the exact rank $r$ of the global optimum $X^*$ (implying $\nabla f(X^*)=0$) and $\bar{\eta}=1$ implied by $1<\kappa \leq 64(\sqrt{2}-1)$ (where the proof is shown in \cite[Appendix E]{Zeng-SGD2019}).

\begin{lemma}
\label{Lemm:feasibility}
Let Assumption \ref{Assump:objfun} hold with $1<\kappa\leq 64(\sqrt{2}-1)$. For any $U \in \mathbb{R}^{p\times r}$, let $X = UU^T$. If ${\cal E}(U,U_r^*) < \gamma_0 \sigma_r(X_r^*)$,
where $\gamma_0 = (\sqrt{2}-1)\kappa^{-1}$,
then the following hold:
\begin{enumerate}
\item[(a)] $\|\nabla f(X)\|_F \leq  (2\sqrt{\gamma_0}+\gamma_0)L\tau(U_r^*)\sigma_r(X^*),$

\item[(b)] $\bar{X}:=X - \frac{1}{L}  {\cal P}_U\nabla f(X){\cal P}_U$ is symmetric and positive semidefinite with rank $r$,

\item[(c)] $ ({\bf I}-{\cal P}_U)X^* =0.$
\end{enumerate}
\end{lemma}

\section*{Appendix B. Some Key Lemmas}
\label{app:proof-of-key-lemma}

Here, we develop two lemmas to bound these two terms, i.e., $2{\mathbb E}_{i_t}[\langle v_k^tU^t, U^t - V_{U^t}^*\rangle]$ and ${\mathbb E}_{i_t}[\|v_k^tU^t\|_F^2]$ in (\ref{Eq:keyineq1}), respectively.

\subsection*{Appendix B.1. Bound $2{\mathbb E}_{i_t}[\langle v_k^tU^t, U^t - V_{U^t}^*\rangle]$}
\label{app:first-term}

\begin{lemma}
\label{Lemm:bound-innerproduct}
Let Assumption \ref{Assump:objfun} hold with $1<\kappa \leq 64(\sqrt{2}-1)$. For any $U\in \mathbb{R}^{p\times r}$, let $X = UU^T$, $Q_U$ be a basis of the column space of $U$, and ${\cal P}_U := Q_UQ_U^T$. If ${\cal E}(U^t,U_r^*) < \gamma_0 \sigma_r(X^*)$ with $\gamma_0 = (\sqrt{2}-1)\kappa^{-1}$,
there holds
\begin{align}
&2{\mathbb E}_{{\cal I}_t} [\langle v_k^t U^t, U^t-V_{U^t}^* \rangle] \label{Eq:T1-lowerbound}\\
&\geq (\sqrt{2}-1)\mu \sigma_r(X^*) {\cal E}(U^t,U_r^*)
- L{\cal E}^2(U^t,U_r^*)  + \frac{1}{4L} \|{\cal P}_{U^t}\nabla f(X^t)\|_F^2.   \nonumber
\end{align}
\end{lemma}

\begin{proof}
Note that
\begin{align}
&2 {\mathbb E}_{{\cal I}_t}[\langle v_k^t U^t, U^t-V_{U^t}^* \rangle]
= 2 \langle \nabla f(X^t), X^t-V_{U^t}^* (U^t)^T \rangle  \nonumber\\
& = \langle \nabla f(X^t), X^t-X^* \rangle + \langle \nabla f(X^t), X^t + X^* - 2 V_{U^t}^* (U^t)^T \rangle.
\label{Eq:T1-lowerbound0}
\end{align}
In the following, we respectively establish the lower bounds of these two terms in the right-hand side of the above equality.

\textbf{(a) Lower bound of $\langle \nabla f(X^t), X^t-X_r^* \rangle$.}
To bound $\langle \nabla f(X^t), X^t-X_r^* \rangle$, we utilize the following two inequalities mainly by the $L$-smoothness and $(\mu,r)$-restricted strong convexity of $f$, that is,
\begin{align*}
\text{(i)} & \ f(X^*) \geq f(X^t) + \langle \nabla f(X^t), X^*-X^t \rangle + \frac{\mu}{2}\|X^* - X^t\|_F^2, \\
\text{(ii)} & \ f(X^t) \geq f(X^*) + \frac{1}{2L}\cdot\|{\cal P}_{U^t}\nabla f(X^t)\|_F^2,
\end{align*}
where (i) holds for the $(\mu,r)$-restricted strong convexity of $f$, (ii) holds for the following inequality induced by the $L$-smoothness of $f$, i.e.,
\begin{align*}
f(X^t)
&\geq f(\bar{X}^t) + \langle \nabla f(X^t), X^t-\bar{X}^t \rangle - \frac{L}{2}\|X^t-\bar{X}^t\|_F^2\\
&(\text{where}\ \bar{X}^t: = X^t - \frac{1}{L}{\cal P}_{U^t}\nabla f(X^t){\cal P}_{U^t}) \\
&= f(\bar{X}^t) +\frac{1}{2L} \|{\cal P}_{U^t}\nabla f(X^t)\|_F^2\\
&(\because \langle \nabla f(X^t), {\cal P}_{U^t} \nabla f(X^t){\cal P}_{U^t}\rangle = \|{\cal P}_{U^t} \nabla f(X^t)\|_F^2),
\end{align*}
and $f(\bar{X}^t) \geq f(X^*)$ since $X^*$ is an optimum and $\bar{X}^t$ is a feasible point by Lemma \ref{Lemm:feasibility}(b).
Summing the inequalities (i)-(ii) yields
\begin{align}
&\langle \nabla f(X^t), X^t-X^* \rangle
\geq \frac{\mu}{2} \|X^t-X^*\|_F^2
+ \frac{1}{2L} \|{\cal P}_{U^t}\nabla f(X^t)\|_F^2 \nonumber\\
&\geq (\sqrt{2}-1)\mu \sigma_r(X^*) {\cal E}(U^t,U_r^*)
+\frac{1}{2L} \|{\cal P}_{U^t} \nabla f(X^t)\|_F^2,
\label{Eq:T11}
\end{align}
where the second inequality is due to Lemma \ref{Lemm:dist-twomat}, i.e., $\|X^t-X^*\|_F^2 \geq 2(\sqrt{2}-1)\sigma_r(X^*) {\cal E}(U^t,U_r^*)$.

\textbf{(b) Lower bound of $\langle \nabla f(X^t), X^t + X^* - 2 V_{U^t}^* (U^t)^T \rangle$.}
Note that
\begin{align}
& \langle \nabla f(X^t), X^t + X^* - 2 V_{U^t}^* (U^t)^T \rangle \nonumber\\
& = \langle {\cal P}_{U^t} \nabla f(X^t) + ({\bf I}-{\cal P}_{U^t})\nabla f(X^t),  X^t + X^* - 2 V_{U^t}^* (U^t)^T \rangle \nonumber\\
& = \langle {\cal P}_{U^t} \nabla f(X^t), X^t + X^* - 2 V_{U^t}^* (U^t)^T \rangle \nonumber\\
& =\langle {\cal P}_{U^t} \nabla f(X^t), (U^t-V_{U^t}^*)(U^t-V_{U^t}^*)^T \rangle \nonumber\\
& \geq -\frac{1}{4L} \|{\cal P}_{U^t}\nabla f(X^t)\|_F^2
- L{\cal E}^2({U^t},U_r^*), \label{Eq:T12}
\end{align}
where the second equality is due to $\langle ({\bf I}-{\cal P}_{U^t})\nabla f(X^t), X^t \rangle =0$, $\langle ({\bf I}-{\cal P}_{U^t})\nabla f(X^t), V_{U^t}^* (U^t)^T \rangle =0$ and $\langle ({\bf I}-{\cal P}_{U^t})\nabla f(X^t), X^* \rangle =0$ by $({\bf I}-{\cal P}_{U^t})U^t = 0$ and Lemma \ref{Lemm:feasibility}(c),
and the third equality holds for $X^* = U_r^* {U_r^*}^T = (U_r^*R_{U^t}^*)(U_r^*R_{U^t}^*)^T = V_{U^t}^* {V_{U^t}^*}^T$, and the inequality holds for the basic inequality: $\langle Y, Z \rangle \geq -\frac{\delta}{2} \|Y\|_F^2 - \frac{1}{2\delta} \|Z\|_F^2$ for any $Y, Z \in \mathbb{R}^{p\times p}$ and $\delta= \frac{1}{2L}$.

Thus, substituting \eqref{Eq:T12} and \eqref{Eq:T11} into \eqref{Eq:T1-lowerbound0}
yields the claim of this lemma.
\end{proof}

\subsection*{B.2. Bound ${\mathbb E}_{{\cal I}_t}[\|v_k^tU^t\|_F^2]$}

\begin{lemma}
\label{Lemm:bound-variance}
Let Assumption \ref{Assump:objfun} hold. Assume that ${\cal E}(U^t,U_r^*)<\gamma_0 \sigma_r(X^*)$ and ${\cal E}(\tilde{U}^k,U_r^*)<\gamma_0 \sigma_r(X^*)$, then
\begin{align}
{\mathbb E}_{{\cal I}_t}[\|v_k^t U^t\|_F^2]
\leq {\frac{2}{b}}(2+\sqrt{\gamma_0})^2 L^2 \|X^*\|_2 B ({\cal E}(U^t,U_r^*)
+{\cal E}(\tilde{U}^k,U_r^*)) + \|{\cal P}_{U^t} \nabla f(X^t)\|_F^2 B, \label{Eq:bound-variance}
\end{align}
where $\gamma_0 = (\sqrt{2}-1)\kappa^{-1}$, and $B$ is specified in \eqref{Eq:B}.
\end{lemma}

\begin{proof}
Let $\xi_{i_t}:= \left(\nabla f_{i_t}(X^t) - \nabla f_{i_t}(\tilde{X}^k) + \nabla f(\tilde{X}^k)\right)U^t$,
then $\{\xi_{i_t}\}_{i_t\in {\cal I}_t}$ are independent and identically distributed (i.i.d) with $\mathbb{E}_{i_t}[\xi_{i_t}] = \nabla f(X^t)U^t$.
Note that the definition of $v_k^t$, we have
\begin{align}
&\mathbb{E}_{{\cal I}_t}[\|v_k^t U^t\|_F^2] = \frac{1}{b^2}\mathbb{E}_{{\cal I}_t} [\|\sum_{i_t\in {\cal I}_t} \xi_{i_t}\|_F^2] \nonumber\\
&=\frac{1}{b^2} \left( \mathbb{E}_{{\cal I}_t} [\|\sum_{i_t\in {\cal I}_t} \xi_{i_t} -b\nabla f(X^t)U^t \|_F^2] + b^2 \|\nabla f(X^t)U^t \|_F^2\right) \nonumber\\
&=\frac{1}{b} \mathbb{E}_{i_t}[\|\xi_{i_t} - \nabla f(X^t)U^t\|_F^2]+\|\nabla f(X^t)U^t\|_F^2. \label{Eq:vkt-bound1}
\end{align}
where the second equality holds for $\mathbb{E}[\|\xi\|^2] = \mathbb{E}[\|\xi-\mathbb{E}\xi\|^2] + \|E\xi\|^2$ for some random variable $\xi$,
and the final equality holds for $\mathbb{E}[\|\sum_{i=1}^b (\xi_i-\mathbb{E} \xi_i)\|^2] = b \mathbb{E}[\|\xi_i-\mathbb{E} \xi_i\|^2]$ for $b$ i.i.d. random variables.
Note that
\begin{align}
&\mathbb{E}_{i_t}[\|\xi_{i_t} - \nabla f(X^t)U^t\|_F^2] \nonumber\\
&= \mathbb{E}_{i_t}[\|(\nabla f_{i_t}(X^t) - \nabla f_{i_t}(\tilde{X}^k))U^t -( \nabla f(X^t)-\nabla f(\tilde{X}^k))U^t\|_F^2]\nonumber\\
&\leq \mathbb{E}_{i_t}[\|(\nabla f_{i_t}(X^t) - \nabla f_{i_t}(\tilde{X}^k))U^t\|_F^2]\nonumber\\
&\leq L^2 \|X^t-\tilde{X}^k\|_F^2 \|X^t\|_2\nonumber\\
&\leq 2L^2 (\|X^t-X^*\|_F^2 +\|\tilde{X}^k-X^*\|_F^2) B, \label{Eq:vkt-bound2}
\end{align}
where the first inequality holds for $\mathbb{E}[\|\xi-\mathbb{E} \xi\|^2] \leq \mathbb{E}[\|\xi\|^2]$ for any random variable $\xi$,
the second inequality holds for Assumption \ref{Assump:objfun} and the norm inequality $\|AB^T\|_F \leq \|A\|_2\|B\|_F$ for any two matrices $A, B$ of the same sizes with $A$ being full-column rank,
and the third inequality holds for the basic inequality $(a-c)^2 \leq 2(a-d)^2 + 2(c-d)^2$ and the definition of $B$.
Plugging \eqref{Eq:vkt-bound2} into \eqref{Eq:vkt-bound1} and noting that $\|\nabla f(X^t)U^t\|_F^2\|\leq \|{\cal P}_{U^t}\nabla f(X^t)\|_F^2 \|X^t\|_2\leq \|{\cal P}_{U^t}\nabla f(X^t)\|_F^2 B$ yields
\begin{align*}
\mathbb{E}_{{\cal I}_t}[\|v_k^t U^t\|_F^2]
&\leq \frac{2}{b}L^2 (\|X^t - X^*\|_F^2 + \|\tilde{X}^k - X^*\|_F^2)B
+ \|{\cal P}_{U^t}\nabla f(X^t)\|_F^2 B \nonumber\\
&\leq \frac{2}{b}(2+\sqrt{\gamma_0})^2L^2 \|X^*\|_2 B({\cal E}(U^t,U_r^*)+{\cal E}(\tilde{U}^k,U_r^*))
+ \|{\cal P}_{U^t}\nabla f(X^t)\|_F^2 B,
\end{align*}
where the second inequality holds for Lemma \ref{Lemm:X-Xr*}. This finishes the proof of this lemma.
\end{proof}
As shown by the proof of this lemma, the introduced semi-stochastic gradient is very crucial to the establishment of the key inequality \eqref{Eq:vkt-bound2} for bounding the term ${\mathbb E}_{{\cal I}_t}[\|v_k^tU^t\|_F^2]$.

\section*{Appendix C. Proof of Proposition \ref{Propos:initial-projgrad}}
\label{app:proof-propos-initial-projgrad}

The proof of Proposition \ref{Propos:initial-projgrad} is similar to that of \cite[Proposition 1]{Zeng-SGD2019}. We provide its proof here mainly for the completeness.

\begin{proof}
From \cite[Theorem 3.6]{Bubeck2014}, we have that for consecutive updates $\tilde{X}^{t-1}$, $\tilde{X}^t$, and the optimum $X^*$, projected gradient descent satisfies:
\begin{align*}
\|\tilde{X}^t - X^*\|_F^2 \leq (1-\kappa^{-1})\cdot \|\tilde{X}^{t-1}-X^*\|_F^2,
\end{align*}
which implies for any $t \in \mathbb{N}$,
\[
\|\tilde{X}^t - X^*\|_F^2 \leq (1-\kappa^{-1})^t\cdot \|\tilde{X}^{0}-X^*\|_F^2 = (1-\kappa^{-1})^t\cdot \|X^*\|_F^2.
\]
Thus, by the hypothesis of $T$,
\[
\|X^0 - X^*\|_F^2 = \|\tilde{X}^T - X^*\|_F^2 \leq \frac{16(\sqrt{2}-1)^2 \sigma_{r}^2(X^*)}{9\kappa}.
\]
By Lemma \ref{Lemm:dist-twomat}, there holds
\begin{align*}
{\cal E}(\tilde{U}^0,U^*)
\leq \frac{\|X^0-X^*\|_F^2}{2(\sqrt{2}-1)\sigma_r(X^*)}
\leq \frac{8(\sqrt{2}-1) \sigma_r(X^*)}{9\kappa}.
\end{align*}
Therefore, we end the proof.
\end{proof}

\section*{Acknowledgment}
The work of Jinshan Zeng is supported in part by the National Natural Science Foundation (NNSF) of China (No. 61977038), and Thousand Talents Plan of Jiangxi Province (No. jxsq2019201124). The work of Ke Ma is supported in part by the NNSF of China (No. 62006217), and China Post-doctoral Science Foundation (2020M680651), as well as supported by the Fundamental Research Funds for Central Universities. The research of Yuan Yao is supported in part by HKRGC 16303817, ITF UIM/390, as well as awards from Tencent AI Lab and Si Family Foundation. Part of Jinshan Zeng's work was done when he visited at Liu Bie Ju Centre for Mathematical Sciences, City University of Hong Kong.

\vskip 0.2in
\bibliography{sample}

\end{document}